\documentclass[12pt]{amsart}
\usepackage{amsmath, amssymb, amsthm}
\usepackage{mathtools}
\usepackage{mathptmx}
\usepackage{tikz}
\usepackage{hyperref}
\definecolor{darkblue}{RGB}{0,0,160}
\hypersetup{
colorlinks,%
citecolor=darkblue,%
filecolor=black,%
linkcolor=darkblue,%
urlcolor=darkblue
}
\usepackage[centering,
			textwidth=15cm,
			textheight=22cm,
			top=3.5cm,
			footskip=20pt,
			marginparwidth=2.9cm
			]{geometry}
\usepackage{xspace}
\usepackage[left]{lineno}

\DeclareMathOperator{\cone}{cone}
\DeclareMathOperator{\conv}{conv}
\DeclareMathOperator{\cut}{Cut}

\DeclareMathOperator{\gp}{gp}

\DeclareMathOperator{\id}{id}
\DeclareMathOperator{\inte}{int}

\DeclareMathOperator{\modulo}{mod}
\DeclareMathOperator{\reg}{reg}
\DeclareMathOperator{\Tor}{Tor}
\def\+#1{\relax\ifmmode\if\noexpand #1\relax \mathop{\kern
    0pt^+{#1}}\nolimits\else \kern 0pt^+\!#1 \fi\else$^*$#1\fi}
\newcommand{\Kb}{\mathbb{K}}
\newcommand{\Nb}{\mathbb{N}}
\newcommand{\Rb}{\mathbb{R}}
\newcommand{\Qb}{\mathbb{Q}}
\newcommand{\Zb}{\mathbb{Z}}
\newcommand{\zerobf}{\mathbf{0}}
\newcommand{\onebf}{\mathbf{1}}
\newcommand{\twobf}{\mathbf{2}}
\newcommand{\abf}{\mathbf{a}}
\newcommand{\bbf}{\mathbf{b}}

\newcommand{\vbf}{\mathbf{v}}
\newcommand{\wbf}{\mathbf{w}}
\newcommand{\xbf}{\mathbf{x}}
\newcommand{\ybf}{\mathbf{y}}
\newcommand{\zbf}{\mathbf{z}}
\newcommand{\MC}{\textsc{MaxCut}\xspace}
\numberwithin{equation}{section}

\newtheoremstyle{thm}{}{}%
     {\em}
     {}
     {\bf}
     {.}
     {0.5em}
     {\thmname{#1}\thmnumber{ #2}\thmnote{ #3}}

\newtheoremstyle{def}{}{}%
     {\rm}
     {}
     {\bf}
     {.}
     {0.5em}
     {\thmname{#1}\thmnumber{ #2}\thmnote{ #3}}

\theoremstyle{thm}
\newtheorem{theorem}{Theorem}[section]
\newtheorem{lemma}[theorem]{Lemma}
\newtheorem{corollary}[theorem]{Corollary}
\newtheorem{proposition}[theorem]{Proposition}
\newtheorem{conjecture}[theorem]{Conjecture}
\newtheorem{problem}[theorem]{Problem}
\theoremstyle{def}
\newtheorem{definition}[theorem]{Definition}
\newtheorem{example}[theorem]{Example}

\title{Seminormality, canonical modules, and regularity of  cut polytopes}
\author{Mitra Koley}
\address{
Tata Institute of Fundamental Research,
School of Mathematics,
Mumbai, India
}
\email{mitrak@math.tifr.res.in}

\author{Tim R\"omer}
\address{
Osnabr\"uck University,
Institute of Mathematics,
Osnabr\"uck, Germany
}
\email{troemer@uos.de}

\thanks{
The first author was supported by the DAAD programme Research Stays for University Academics and Scientists in 2019, grant number 57442043.
}

\begin{document}
\begin{abstract}
Motivated by a conjecture of Sturmfels and Sullivant
we study normal cut polytopes. After a brief survey of known results for normal cut polytopes it is in particular observed that
for simplicial and simple cut polytopes
their cut algebras are normal and hence Cohen--Macaulay.
Moreover, seminormality is considered. It is shown that the cut algebra of $K_5$ is not seminormal which implies again the known fact that it is not normal.
For normal Gorenstein cut algebras and other cases of interest we determine their canonical modules.
The Castelnuovo--Mumford regularity of a cut algebra is
computed for various types of graphs and bounds for it are provided if normality is assumed.
As an application we classify all graphs for which the cut algebra has regularity less than or equal to 4.

\end{abstract}
\maketitle
\section{Introduction}
The \MC problem in combinatorial optimization
is one of the 21 NP-complete problems of Karp \cite{Karp}
with  interesting applications. For an overview of these we refer to \cite{Deza-Laurent-I, Deza-Laurent-II}. Closely related to this is the polyhedral point of view of the story, i.e.\ the study of the \emph{cut polytope} $\cut^{\Box}(G)$ of a graph $G$ and its geometric properties (see Section \ref{Section-2} for definitions). Facts based on cases of graphs where polynomial time algorithms for \MC are known
or other special knowledge exists
(see, e.g.,
\cite{Barahona-83, BoJuReRi, Kaminski})
led to a larger number of beautiful results
where all or at least many
facet defining inequalities of cut polytopes
and related objects  can be explicitly described.
See
\cite{Barahona-86, ChJKNoRo, Deza-Dutour-20, Neta},
the book of Deza--Laurent \cite{Deza-Laurent-10}
and also
\cite{ChRe, Deza-Dutour-16} for statements on
enumerations of facets.

The connection to algebraic geometry and commutative
algebra was initiated by Sturmfels and Sullivant, who introduced the
\emph{cut algebra} $\Kb[G]$ over a field $\Kb$ and its defining \emph{cut ideal} $I_G$ for a graph $G$ (see \cite{Sturmfels-Sullivant}). Observe that $\Kb[G]$ is the toric algebra associated to the $0/1$-polytope $\cut^{\Box}(G)$. See Section \ref{Section-2} for these notations
and \cite{Bruns-Gubeladze} as a general reference on toric algebra and geometry. Note that in \cite{Sturmfels-Sullivant}
also applications of cut objects to algebraic statistics
are discussed.
From purely algebraic prospects several conjectures
of that paper are of great interest,
from which we summarize the following combined one:
\begin{conjecture}{(\cite{Sturmfels-Sullivant})}
\label{Conj-StSu}
The following statements are equivalent:
\begin{enumerate}
\item
$G$ is $K_5$-minor-free;
\item
$I_G$ is generated in degrees less than or equal to $4$;
\item
$\Kb[G]$ is Cohen--Macaulay;
\item
$\Kb[G]$ is normal.
\end{enumerate}
\end{conjecture}

In \cite{Sturmfels-Sullivant} it was already observed that
for (ii), (iv) it is necessary that (i) holds. By a theorem of Hochster \cite{Hochster}
one knows that (iv) implies (iii). Further partial
results are discussed in \cite{Sturmfels-Sullivant}
as well as in \cite{E, NP, Ohsugi, Romer-Madani}.
The characterizations of $\Kb[G]$ being a complete intersection in  \cite[Theorem 6.9]{Romer-Madani}
and of $\Kb[G]$ being a normal Gorenstein algebra
in \cite[Theorem 3.4]{Ohsugi-Gorenstein} yield further support for the conjecture. Beside these facts Conjecture \ref{Conj-StSu} is open and remains a challenging
task. Additionally to the mentioned work,
further properties of $\Kb[G]$
and $I_G$ were studied in particular in
\cite{Potka-Sarmiento, Sakamoto, Shibata}.

The main goal of this paper is to study
$\Kb[G]$ in the normal situation
and to consider algebraic properties of interest. We start in Section \ref{Section-3} with a brief survey of known facts  related to Conjecture \ref{Conj-StSu}, which deepens the previous discussion. Our first contribution is in Corollary \ref{ring-graph}
the observation that using results of \cite{Ohsugi} one can give a new short proof that for a ring graph $G$ the algebra $\Kb[G]$ is normal.  This is interesting since the first proof of this in \cite{NP} contains an error as was pointed out in \cite{Sakamoto}. We conclude Section \ref{Section-3}
by discussing examples where $G$ is an outerplanar graph
and where $\cut^{\Box}(G)$ is a simplicial or  simple polytope, respectively. We expect that further polytopal properties of $\cut^{\Box}(G)$ should provide further support for Conjecture \ref{Conj-StSu}.

A natural approach to relax the property normality
is to study seminormal algebras.
See \cite{BoEiNi, Bruns-Gubeladze, BrPiRo, GeKaRe, Hochster-Roberts, Nguyen, Nitsche,  Swan, Traverso, Yanagawa} for results related to seminormality
in commutative algebra and related areas.
In Section \ref{Section-4}, we present in Theorem \ref{Theorem-K5} a proof
that $\Kb[K_5]$ is not seminormal,
which complements the discussion in \cite[Example 7.2]{Romer-Madani} and reproves also some (computer based) facts in \cite[Table 1]{Sturmfels-Sullivant}.
If a cut algebra is seminormal, then the underlying graph has to be $K_5$-minor-free (see Corollary \ref{Cor-seminormal}).
Remarkably, using this fact and other clever arguments related to very ampleness,
Laso\'{n} and Micha{\l}ek \cite{Lason-Michalek}
were recently able to prove that $\Kb[G]$ is seminormal if and only if
it is normal. An interesting consequence of this fact is that
results of Ohsugi \cite{Ohsugi} hold in the seminormal situation
like that seminormality is a minor-closed property of the underlying graphs.


For a Cohen--Macaulay algebra there is always an associated canonical module, which is of great importance for the algebra itself (see \cite{Bruns-Gubeladze, Bruns-Herzog} for details).
The canonical module of cut algebras has not been studied before. We determine this module for normal Gorenstein cut algebras of graphs in Theorem \ref{Theorem-canonical-module}
and other cases of interest.

The remaining part of this work is devoted in Section \ref{Section-6} to study
the Castelnuovo--Mumford regularity of cut algebras.
For arbitrary standard graded algebras this
is one of its key homological invariants.
First results related to this invariants
were obtained in \cite{NP} and \cite{Romer-Madani}.
See also \cite{Potka-Sarmiento} for relevant facts.
We compute the Castelnuovo--Mumford regularity
for various types of graphs including ring graphs (see Theorem \ref{regularity-ring-graph})
and
provide upper as well as lower bounds of it
for all normal cut algebras. Extending the classification in \cite{Romer-Madani} of graphs for which the cut algebra has regularity 0 or 1, we classify in Theorem \ref{Theorem-classify-reg-4} all graphs
for which their cut algebra has regularity less than or equal to 4.

Finally, we formulate some research problems in Section \ref{Section-7}.

The authors are grateful to Winfried Bruns
for inspiring discussions in general and especially related
to computations with Normaliz \cite{Bruns-Normaliz}. We thank the anonymous referees for their very helpful comments and constructive remarks on this manuscript.
\section{Preliminaries}
\label{Section-2}

In this section we recall basic facts and notation used in the following. For further details to graph theory we refer to the book of Diestel \cite{Diestel}. A general reference for monoids, polytopes, their algebras and properties of them is the book of Bruns and Gubeladze \cite{Bruns-Gubeladze}.

\subsection{Graphs}
Let $G=(V, E)$ be a graph with \emph{vertex set} $V(G)=V\neq \emptyset$ and \emph{edge set} $E(G)=E$. We always consider undirected and simple  graphs (i.e.\ without multiple edges and without loops). An edge $\{v,w\}\in E$ is also denoted by $vw$.

A graph $H$ is a \emph{subgraph} of a graph $G$ if $V(H)\subseteq V(G)$ and $E(H)\subseteq E(G)$. The subgraph is called \emph{induced} if $E(H)$
contains all possible edges $vw \in E(G)$ with $v,w\in V(H)$. In this case we also denote
$H$ by $G_W$ where $\emptyset \neq W=V(H)\subseteq V(G)$.

A \emph{cycle} $C=C_n$ \emph{of length $n$} is a graph with $n$ vertices $\{v_1,\ldots,v_n\}$ and  edges $\{v_iv_{i+1}: 1\leq i\leq n-1\}\cup \{v_nv_1\}$. A \emph{triangle} is a cycle of length $3$. A \emph{cycle of a graph} $G$ is a cycle $C$ of some length which is a subgraph of $G$. A \emph{chord} of such a cycle is an edge $vw\in E(G)\setminus E(C)$ with $v,w \in V(C)$. A cycle of a graph is induced if and only if it is a cycle without chords. A graph is said to be \emph{chordal} if all its induced cycles are triangles. An edge $e$ of a graph is called a \emph{bridge} if there does not exist a cycle which contains $e$.

A \emph{complete} graph with $n$ vertices
is a graph $G$ with $|V(G)|=n$ and   $E(G)=\{vw: v,w\in V(G),\ v\neq w\}$. We denote this graph by $K_n$. An
induced complete subgraph of a graph $G$ is also called a \emph{clique} of it. A graph $G$ is a \emph{bipartite} graph if there exists a partition $V(G)= V_1\cup V_2$  with $V_1\cap V_2=\emptyset$ such that any edge of $G$ is of the type $vw$ with $v\in V_1$, $w\in V_2$. Recall that $G$ is bipartite if and only if $G$ has no cycle of odd length. A \emph{complete bipartite graph} is a bipartite graph $G$ as above where  any $v\in V_1$, $w\in V_2$ yield an edge $vw \in E(G)$.
If $|V_1|=m$, $|V_2|=n$, then
this type of graph is also denoted by $K_{m,n}$.

\emph{Edge deletion} and \emph{edge contraction}
of a graph $G$ at an edge $e\in E(G)$
are denoted by $G\setminus e$ and $G/e$.
A graph is said to be a \emph{minor} of $G$ if it can be obtained from it by a sequence of edge deletions and edge contractions. Recall that if $K_n$ is a minor of $G$, then $K_n$ can be obtained from $G$ by using  only edge contractions. By Kuratowski's Theorem a graph is \emph{planar} if and only if it is $K_5$- and $K_{3,3}$-minor-free (after removing isolated vertices). Maximal planar graphs are called \emph{triangulations}. An
\emph{outerplanar} graph is a graph which has a planar drawing such that all vertices belong to the outer face of that drawing.

Let $G_1$ and $G_2$ be two graphs such that $H=(G_1)_{V(G_1)\cap V(G_2)}=(G_2)_{V(G_1)\cap V(G_2)}$. The new graph $G_1\# G_2=G_1\#_H G_2$ with vertex set $V(G_1)\cup V(G_2)$ and edge set $E(G_1)\cup E(G_2)$ is called the \emph{$H$-sum} of $G_1$ and $G_2$. If
$H=K_{n+1}$, then $G_1\#G_2=G_1\#_{K_{n+1}}G_2$ is called an \emph{$n$-clique-sum} or simply an \emph{$n$-sum} of graphs.

\subsection{Polytopes and Cones}
A \emph{hyperplane} $H=H_{\abf,b}=\{\xbf\in \Rb^d: \abf\cdot \xbf-b=0\}$,
where $\abf\in \Rb^d$, $b\in \Rb$, induces two (closed) half-spaces $H^+$ and $H^-$ defined by $\abf\cdot\xbf-b\geq 0$ and $\abf\cdot \xbf-b\leq 0$. \emph{Polyhedra} are intersections of finitely many half-spaces. Thus, such a polyhedron $P$ is the set of solutions of a system of inequalities, which we write as $A\xbf \leq \mathbf{b}$ where $A$ is a real-valued $n\times d$-matrix for some $n\in \Nb$ , $\xbf\in \Rb^d$ and $\bbf\in  \Rb^n$.
By Weyl--Minkowski, bounded polyhedra are exactly the \emph{polytopes},
i.e.\ convex hulls $\conv(\xbf_1,\dots,\xbf_m)$ of finitely many points in $\xbf_1,\dots,\xbf_m\in\Rb^d$. If a polyhedron $P$ is defined by linear inequalities, i.e.\ $\mathbf{b}=0$,
then we obtain all finitely generated (convex) \emph{cones} of the form $\cone(\ybf_1,\dots,\ybf_m)=\Rb_{\geq 0}\ybf_1+\dots+\Rb_{\geq 0}\ybf_m$ for $\ybf_1,\dots,\ybf_m\in\Rb^d$.

Let $P$ be a polyhedron.
Its \emph{dimension} $\dim P$ is the dimension of the affine hull of $P$.
A hyperplane $H=H_{\abf,b}$ is said to be a \emph{support hyperplane} of $P$ if all  $\xbf\in P$ satisfy $\abf\cdot\xbf\leq b$ and $H \cap P \neq \emptyset$.
The intersection $F = H \cap P$ is called a \emph{face} of $P$ and $H$ is also said to be a \emph{support hyperplane} associated with $F$. Observe that here $P$ is also a face of itself.
Note that faces of polytopes are again polytopes
and faces of cones are cones.
Faces of dimension $0$ and  $1$ are called \emph{vertices} and \emph{edges}, respectively. A face of dimension $\dim P -1$ is a \emph{facet} of $P$. Then the inequality $\abf\cdot\xbf\leq b$ of the support hyperplane is also said to be \emph{facet defining}.

We introduce further notation of polytopes. A polytope of dimension $d$ is a  \emph{simplex} if it is  the convex hull of $d + 1$ affinely independent points. It is \emph{simple} if each vertex of it is contained in exactly $d$ facets. We say that a polytope is \emph{simplicial} if each facet is a simplex. A polytope $P$ is a \emph{lattice polytope}
if its vertices are \emph{lattice points} of the integral lattice $\Zb^d\subseteq \Rb^d$. The set of all lattice points in $P\cap \Zb^d$ of $P$ is written as $L_P$.

\subsection{Monoids and their algebras}

All \emph{monoids} in this paper are commutative, written additively and contain a neutral element. Let $\gp(M)$ be the abelian group induced by $M$. A subset $I$ of $M$ is called an \emph{ideal} of $M$ if $M + I \subseteq I$.

A monoid  $M$ is \emph{affine} if it is finitely generated and isomorphic to a submonoid of some $\Zb^n$.
Let $M\subseteq \Zb^n$ be an affine monoid.  The  (relative) \emph{interior} of $M$ is defined by $\inte(M) = M \cap \inte(\cone(M))$, which is an ideal of $M$ and where $\inte(\cone(M))$ denotes the (relative) interior of the conical hull $\cone(M)$ of $M$. For a face $F$ of $\cone(M)$ the set $F\cap M$ is a submonoid of $M$, which is called a \emph{face} of $M$.

Given a field $\Kb$ one can associate to a monoid $M$ a $\Kb$-algebra $\Kb[M]$, called its \emph{monoid algebra}, as follows. As a vector space $\Kb[M]$ is free with a basis consisting of the symbols $X^a$, $a \in  M$. The elements $X^a$ are called the \emph{monomials} of $\Kb[M]$. The multiplication on $\Kb[M]$ is defined on the monomials by $X^aX^b = X^{a+b}$ and this extends linearly to all element of $\Kb[M]$.
The \emph{normalization} of a monoid $M$, denoted by $\overline{M}$, is the set
$
\overline{M}=\{x \in\gp(M): nx\in M \text{ for some nonzero } n\in \Nb\}.
$
Note that $\overline{M}$ is again a monoid. An $x\in \overline{M}$ is also said to be \emph{integral} over $M$. We say that $M$ is \emph{normal} if $M=\overline{M}$. The algebra $\Kb[M]$ is normal if and only if $M$ is a normal monoid.

Let $P\subseteq \Rb^d$ be a lattice polytope. The \emph{polytopal monoid} $M_P$ associated with $P$ is the submonoid of $\Rb^{d+1}$ generated by
$(\xbf, 1)$ for $\xbf\in L_P$.  As $L_P$ is finite, $M_P$ is affine. Observe also that for a lattice polytope $P$, the monoid algebra $\Kb[M_P]$, associated to the polytopal  monoid $M_P$,
equals to its \emph{polytopal algebra} $\Kb[P]$ defined as
\[
\Kb[P] = \Kb[\ybf^{\abf}z : \mathbf{a}\in P\cap \Zb^d].
\]
Here $\ybf^{\abf} = y^{a_1}_1\cdot \ldots\cdot y^{a_d}_d$ for
$\abf = (a_1,\cdots, a_d)\in \Zb^d$ and $z$ is another variable. If $P\subseteq \Rb^d_{\geq 0}$, then
this algebra is a $\Kb$-subalgebra of $\Kb[y_1, \cdots, y_d, z]$. It is standard graded induced by setting $\deg(z)=1$
and $\deg(y_i)=0$ for $i=1,\dots,d$. Observe that $\Kb[P]$ is a toric algebra.

\subsection{Cut polytopes and their algebras}
Let $G = (V ,E)$ be a graph. Vectors in $\Rb^E$
are written as $\xbf=(x_e)_{e\in E}$. For any subset $A\subseteq V$ we define its \emph{cut set} as $\cut(A)=\{e\in E: |A\cap e|=1\}$ and its \emph{cut vector} $\delta_A=(\delta_{A,e})_{e\in E} \in  \Rb^{E}$ by
\[
\delta_{A,e}
=
\begin{cases}
1&\text{if } e\in \cut(A),
\\
0&\text{otherwise}.
\end{cases}
\]
We see that cut sets correspond one-to-one to
cut vectors. Note that $\delta_A=\delta_{A^c}$
where $A^c=V\setminus A$.
The \emph{cut polytope} of $G$, denoted by $\cut^{\Box}(G)$, is the convex hull of all $\delta_A$ for $A\subseteq V$. This polytope is full-dimensional
and has at most $2^{|V|-1}$ many vertices with equality if $G$ is connected. Setting $A=\emptyset$ yields that the zero vector is always a vertex.

In general a description of all facet defining inequalities of cut polytopes is not known and one of the main open problems of the field is to get insights to these (see, e.g., \cite{Deza-Laurent-10}). However, for some interesting special cases one can describe all facets. For later use we recall the following
consequence of results of Barahona \cite{Barahona-86} (see also \cite[Theorem 2.3]{BoJuReRi} and \cite[Proposition 2.1]{Ohsugi}).
\begin{proposition}{(\cite[Section 3]{Barahona-86})}
\label{Prop-facets-K5-minor-free}
Let $G$ be a $K_5$-minor-free graph.
Then $\cut^{\Box}(G)$ is the solution set
of the following system of (valid) inequalities:
\begin{eqnarray*}
0\leq x_e &\leq& 1, \hspace{2pt} \text{ where } e\in E(G) \text{ does not belong to a triangle},
\\
\sum_{f\in F}x_f-
\sum_{e\in E(C)\setminus F}
x_e
&\leq& |F |-1,
\end{eqnarray*}
where $C$ ranges over all induced cycles of $G$ and $F\subseteq E(C)$ with $|F|$ odd.
Moreover, these inequalities define all facets of $\cut^{\Box}(G)$.
\end{proposition}

The affine  monoid $M_G$ associated to $\cut^{\Box}(G)$ is generated by $(\delta_A, 1)$ for  $A\subseteq V$.
As noted in \cite[(1)]{Ohsugi}
it follows from \cite[Page 258]{Laurent}
that for $\xbf\in \Zb^{E}$ and $\alpha\in \Zb$
we have that
\begin{equation}
\label{elt-in-group}
(\xbf,\alpha) \in\gp(M_G)
\textrm{ if and only if }
\sum_{e\in C} x_e \equiv 0
\modulo 2,
\end{equation}
where $C$ ranges over all cycles of $G$.
The following consequence of Proposition \ref{Prop-facets-K5-minor-free} was observed in \cite[Corollary 2.2]{Ohsugi}: let $G$ be a $K_5$-minor-free graph, $\xbf\in \Qb^{E}$
and $\alpha\in \Zb_{\geq 0}$.
Then $(\xbf, \alpha) \in \Qb_{\geq 0} M_G$
if and only if
\begin{eqnarray}
\label{InEq-facets1}
0\leq\ x_e &\leq& \alpha, \hspace{2pt} \text{ where } e\in E(G) \text{ does not belong to a triangle},
\\
\label{InEq-facets2}
\sum_{f\in F}x_f-
\sum_{e\in E(C)\setminus F}
x_e
&\leq& \alpha(|F |-1),
\end{eqnarray}
where $C$ ranges over all induced cycles of $G$ and $F\subseteq E(C)$ with $|F|$ odd.
Note that $(\xbf, \alpha)$ is an element in the interior of $\cone(M_G)$ if and only if all inequalities are strictly satisfied.

The polytopal algebra associated to $\cut^{\Box}(G)$ is
called the \emph{cut algebra} of $G$
and denoted in the following by $\Kb[G]$ for a given field $\Kb$. Observe that
$\cut^{\Box}(G)$ has dimension $|E|$ and thus $\dim\Kb[G]=|E|+1$. Note that the definition
of this algebra in \cite{Sturmfels-Sullivant}
is equivalent to the one given here by \cite[Lemma 4.2]{Romer-Madani}. For a presentation
choose a polynomial ring
\[
S_G=\Kb[q_{A}: A\subseteq V]
\]
where we set $q_{A}=q_{A^c}$. Thus, this ring
is generated only by $2^{|V|-1}$ many variables.
Consider the following surjective
$\Kb$-algebra homomorphism:
\[
\varphi_G \colon S_G \to \Kb[G],\
q_{A} \mapsto \ybf^{\delta_A}z
.
\]
The kernel $I_G$ of $\varphi_{G}$ is a graded ideal
which is  called the \emph{cut ideal} of $G$.
It is a toric ideal and thus generated by (pure) binomials.
Recall
that many algebraic properties of interest of $\Kb[G]$
can be characterized by corresponding monoidal properties of $M_G$ and vice versa.

\section{Normality of cut polytopes}
\label{Section-3}

In this section we present a brief survey of known results related to Conjecture \ref{Conj-StSu}.
Also some new contributions to this conjecture are discussed. In the following we say
that a cut polytope satisfies some algebraic
property if its cut algebra has this property.
For a graph $G$ let $\mu(I_G)$ be the maximal
degree of a minimal (binomial) generator of $I_G$.

It is known that the cut polytope of $K_5$ is not normal, not Cohen--Macaulay and $\mu(I_{K_5})=6$
(see \cite[Table 1]{Sturmfels-Sullivant} for a computational approach and, e.g., \cite[Example 7.2]{Romer-Madani} for a proof).

In \cite[Pages 699--700]{Sturmfels-Sullivant} it was observed that if $\Kb[G]$ is normal, Cohen--Macaulay or
$\mu(I_G)\leq 4$, then the given graph $G$ has to be $K_5$-minor-free (see also \cite[Sections 4--5]{Romer-Madani} for additional remarks).
Here we concentrate on the Cohen--Macaulay part of the conjecture. As mentioned in the introduction, by Hochster's result on affine monoid rings normality
implies Cohen--Macaulayness for cut algebras.
Hence, one key part of the conjecture is to show
that cut algebras of $K_5$-minor-free graphs are normal. Note that in \cite{Sturmfels-Sullivant}  Conjecture \ref{Conj-StSu} is verified computationally for graphs with up to $6$  vertices. In \cite[Theorem 3.2]{sullivant} Sullivant proved that $\cut^{\Box}(G)$ is a compressed polytope if and only if $G$ is $K_5$-minor-free and has no induced cycle of length greater than $4$. One should note that Sullivant's notion of compressed refers to general affine lattices and not just
to $\mathbb{Z}^r$. This differs from the definition of compressed polytopes that has also appeared in the literature. It is known that compressed polytopes are normal, which supports further Conjecture \ref{Conj-StSu} in this case. In particular, this can be applied to $K_5$-minor-free chordal graphs.
Ohsugi proved in \cite[Corollary 2.4 and Theorem 3.2]{Ohsugi} the following two important results with respect to normality of cut polytopes.

\begin{theorem}[(\cite{Ohsugi})]
\label{normality-Ohsugi}
Let $G$ be a graph.
\begin{enumerate}
\item
Let $H$ be a minor of $G$. If $\cut^{\Box}(G)$ is normal, then $\cut^{\Box}(H)$ is normal.
Thus, normality is a minor-closed property.
\item
Let $G=G_1 \#G_2$ be a $0$-, $1$- or $2$-sum of $G_1$ and $G_2$. Then the cut polytope
$\cut^{\Box}(G)$ is normal if and only if the cut polytopes $\cut^{\Box}(G_1)$ and
$\cut^{\Box}(G_2)$ are normal.
\end{enumerate}
\end{theorem}
In \cite[Theorem 3.2]{Ohsugi} the proof of ``the only if'' direction of (ii) relies on \cite[Lemma 3.2(1)]{Sturmfels-Sullivant} which is not correct
as was pointed out in \cite{Romer-Madani}.
Note that all other results in \cite{Ohsugi} use only the ``if part'' of (ii). But this can easily be repaired. In (ii) the graphs $G_1$, $G_2$ are minors of $G_1 \#G_2$, which can be obtained by using only edge contractions (and removing possible multiple edges, loops and isolated vertices, which all do not change cut polytopes). Then one can apply \cite[Lemma 3.2(2)]{Sturmfels-Sullivant} to conclude the proof of (ii).
The following proposition summarizes further related and important results:

\begin{proposition}
\label{known-cases-conjecture}
Let $G$ be a graph.
 \begin{enumerate}
  \item (\cite[Theorem 3.8]{Ohsugi})
  If $G$ is $K_5\setminus e$-minor-free, then
  $\cut^{\Box}(G)$ is normal.
  \item (\cite[Corollary 2.8]{E})
  The cut ideal $I_G$ is generated in degrees less than or equal to $2$ if and only if $G$ is $K_4$-minor-free.
  \item (\cite[Section 4]{KNP})
  Suppose that G has a universal vertex. Then $I_G$ is generated in degrees
less than or equal to $4$ if and only if G is $K_5$-minor-free.
 \end{enumerate}

\end{proposition}

An interesting class of graphs are ring graphs
defined as follows.

Recall that a \emph{cut vertex}
$v\in V(G)$ of a graph $G$ is a vertex
such that induced subgraph on $V(G)\setminus v$
has more connected components than $G$.
A \emph{block} of $G$ is a maximal connected subgraph
of $G$ without cut vertices.
Then $G$ is called a \emph{ring graph} if
any block of $G$ which is not a bridge or a vertex
can be constructed from a cycle by
adding cycles using the operation of taking $1$-sums.
One can see that ring graphs are exactly those
graphs which one can obtain (up to isolated vertices) from trees and cycles
using $0$- or $1$-sums. Moreover, ring graphs are
$K_4$-minor-free. See \cite{NP} for some further details
and their study related to cut objects.
In the latter paper it is stated that cut algebras of ring graphs are Cohen--Macaulay (see \cite[Theorem 6.2]{NP}).
As was pointed out in \cite{Sakamoto} this proof contains a mistake which makes it interesting to give an
alternative argument. Using Theorem \ref{normality-Ohsugi} this is an easy task:

\begin{corollary}
\label{ring-graph}
Let $G$ be a ring graph. Then $\cut^{\Box}(G)$ is normal.
\end{corollary}
\begin{proof}
There are at least two arguments using Ohsugi's results.
For the first one observe that cut polytopes of trees
and cycles are normal as one might check as an exercise.
Then one concludes the proof using Theorem \ref{normality-Ohsugi}(ii) and the definition of ring graphs.
There exists also an even quicker second proof
by applying Proposition \ref{known-cases-conjecture}(i) (which was proved by Ohsugi using Theorem \ref{normality-Ohsugi}(i)), since ring graphs are $K_4$-minor-free graphs.
\end{proof}
Studying this conjecture further and having Kuratowski’s Theorem in mind, one could try to prove it in the case of planer graphs.
But proving the case of planar triangulations is equivalent to showing that any $K_5$-minor-free graph yields a normal cut polytope (see \cite[Conjecture 4.3]{Ohsugi}).

\begin{example}
A partial related result to the latter discussion is
that for example cut polytopes of outerplanar graphs are normal
as one can prove using the results mentioned in this section. Details are left to the reader.
\end{example}


Further evidence to Conjecture \ref{Conj-StSu} are given by the characterizations of $\Kb[G]$ being a complete intersection in  \cite{Romer-Madani}
and being a normal Gorenstein algebra
in \cite{Ohsugi-Gorenstein}.
Special knowledge about polytopes can also help to support \ref{Conj-StSu}.
\begin{example}
\label{thm-simple-simplicial}
Let $G$ be a graph such that the cut polytope  $\cut^{\Box}(G)$ is either
simplicial or simple, then Conjecture \ref{Conj-StSu} is true for $\cut^{\Box}(G)$.  Indeed, the simplicial case is obtained
by \cite[Theorem 4.5]{ChJKNoRo}, where all possible polytopes are listed,
and then using \cite[Examples 3.7]{Ohsugi}
as well as \cite[Table 1]{Romer-Madani} to verify all parts of the conjecture. For simple polytopes
one either applies \cite[Theorem 1]{KW} or
\cite[Lemma 4.2 and Theorem 4.3]{ChJKNoRo}
to obtain a list of possible polytopes, which are $k$-sums
of normal cut polytopes with $k\leq 2$. Then normality follows from
Theorem \ref{normality-Ohsugi}(ii) and $\mu(I_G)\leq 4$ by applying \cite[Theorem 2.1]{Sturmfels-Sullivant}. Note that in both cases normality alone could also be deduced using Proposition \ref{known-cases-conjecture}(i).
\end{example}


\section{Seminormality of cut polytopes}
\label{Section-4}

A natural approach to relax the  normality property
in Conjecture \ref{Conj-StSu} is to study seminormal cut algebras. In particular, see \cite{Bruns-Gubeladze, BrPiRo, Hochster-Roberts} for references of related results to affine monoid rings, which are of relevance in the following. Compared to normality, the definition
of seminormality seems to be less known
and we recall it here in the context of monoids (see, e.g., \cite[Definition 2.39]{Bruns-Gubeladze}):

\begin{definition}
A monoid $M$ is called \emph{seminormal}
if $x \in\gp(M)$ with $2x, 3x\in M$
implies that $x\in M$.
The \emph{seminormalization} $\+M$ of $M$ is the intersection of all seminormal submonoids
of $\gp(M)$ which contain $M$.
\end{definition}
One can see that $\+M$ is a monoid which is affine if
$M$ is affine. We have $M\subseteq \+M$ with equality
if and only if $M$ is seminormal. Note that every normal monoid is seminormal and in general
we have $M\subseteq \+M \subseteq \overline{M}$.
Hochster and Roberts proved in \cite[Proposition 5.32]{Hochster-Roberts} that an affine monoid is seminormal if and only if $\Kb[M]$ is seminormal in the algebraic sense. For an affine monoid $M$, set $M_\ast=\inte(M)\cup \{0\}$. Recall the following result:

\begin{proposition}{(\cite[Proposition 2.20]{Bruns-Gubeladze})}
\label{Prop-seminormal-criterion}
An affine monoid $M$ is seminormal if and only if $(M \cap F)_\ast$ is a normal monoid for every face $F$ of $\cone(M)$.
In particular, if  $M$ is seminormal, then $M_\ast=\overline{M}_\ast$, i.e. $M_\ast$ is normal.
\end{proposition}

In this section we study (non-) seminormal cut algebras
and related questions.
One of our results  is that
the cut polytope of $K_5$ is not seminormal. In particular,
this complements the discussion in
\cite[Example 7.2]{Romer-Madani}
and reproves also some (computer based) facts in \cite[Table 1]{Sturmfels-Sullivant} related to $K_5$.
For sets $\{1,\ldots,n\}$ we write also $[n]$.
\begin{theorem}
\label{Theorem-K5}
The polytope $\cut^{\Box}(K_5)$ is not seminormal. In particular, it is not normal.
\end{theorem}
\begin{proof}
The cut polytope for $K_5$,
$\cut^{\Box}(K_5)$ is the convex hull of the
following cut vectors:
\begin{align*}
\abf_0
&=
(   0,\
    0,\
    0,\
    0,\
    0,\
    0,\
    0,\
    0,\
    0,\
    0
),
&
\abf_1
&=
(   1,\
    1,\
    1,\
    1,\
    0,\
    0,\
    0,\
    0,\
    0,\
    0
),
\\
\abf_2
&=
(
    1,\
    0,\
    0,\
    0,\
    1,\
    1,\
    1,\
    0,\
    0,\
    0
),
&
\abf_3
&=
(
    0,\
    1,\
    0,\
    0,\
    1,\
    0,\
    0,\
    1,\
    1,\
    0
),
\\
\abf_4
&=
(
    0,\
    0,\
    1,\
    0,\
    0,\
    1,\
    0,\
    1,\
    0,\
    1
),
&
\abf_5
&=
(
    0,\
    0,\
    0,\
    1,\
    0,\
    0,\
    1,\
    0,\
    1,\
    1
),
\\
\abf_6
&=
(
    0,\
    1,\
    1,\
    1,\
    1,\
    1,\
    1,\
    0,\
    0,\
    0
),
&
\abf_7
&=
(
    1,\
    0,\
    1,\
    1,\
    1,\
    0,\
    0,\
    1,\
    1,\
    0
),
\\
\abf_8
&=
(
    1,\
    1,\
    0,\
    1,\
    0,\
    1,\
    0,\
    1,\
    0,\
    1
),
&
\abf_9
&=
(
    1,\
    1,\
    1,\
    0,\
    0,\
    0,\
    1,\
    0,\
    1,\
    1
),
\\
\abf_{10}
&=
(
    1,\
    1,\
    0,\
    0,\
    0,\
    1,\
    1,\
    1,\
    1,\
    0
),
&
\abf_{11}
&=
(
    1,\
    0,\
    1,\
    0,\
    1,\
    0,\
    1,\
    1,\
    0,\
    1
),
\\
\abf_{12}
&=
(
    1,\
    0,\
    0,\
    1,\
    1,\
    1,\
    0,\
    0,\
    1,\
    1
),
&
\abf_{13}
&=
(
    0,\
    1,\
    1,\
    0,\
    1,\
    1,\
    0,\
    0,\
    1,\
    1
),
\\
\abf_{14}
&=
(
    0,\
    1,\
    0,\
    1,\
    1,\
    0,\
    1,\
    1,\
    0,\
    1
),
&
\abf_{15}
&=
(
    0,\
    0,\
    1,\
    1,\
    0,\
    1,\
    1,\
    1,\
    1,\
    0
).
\end{align*}
The cut vector $\abf_0$ corresponds to the choice $\emptyset\subseteq V$.
The cut vectors $\abf_1, \ldots,\abf_5$ are induced by $\{1\}, \ldots, \{5\}$. These cut vectors have exactly $4$ nonzero entries as each vertex belongs to $4$ edges. The other cut vectors correspond to the cardinality $2$ subsets of $[5]$ and they have $6$ nonzero entries.
Set
\[
M=
\biggl\{\sum_{i=0}^{15}z_i
(\abf_i, 1):
z_i\in \Zb_{\geq 0}
\biggr\}.
\]
Observe that $\cut^{\Box}(K_5)$ is seminormal if and only if the $M$ is a seminormal affine monoid.

Consider the element
\[
\xbf=(2,\ 2,\ 2,\ 2,\ 2,\ 2,\ 2,\ 2,\ 2,\ 2,\ 4)
    \in \Zb^{11}.
\]
We claim that $\xbf\in \overline{M}_\ast\setminus M_\ast$,
where $\overline{M}_\ast$ is the normalization of $M_\ast$.
Then $M_\ast$ is not normal and Proposition \ref{Prop-seminormal-criterion} yields that $M$ cannot be seminormal.

At first one sees that $\xbf\in \gp(M)$, since, e.g.,
$
\xbf=
\sum_{i=1}^5
(\abf_i, 1)
-
(\abf_0, 1)
.
$ Thus, $\xbf\in \gp(M_\ast)$ by \cite[Corollary 2.25]{Bruns-Gubeladze}. Note that
\[
4\xbf=\sum_{i=0}^{15}
(\abf_i, 1)
\in \inte(M)\subseteq M_\ast.
\]
That $4\xbf\in\inte(M)$ can, e.g., be deduced from the fact that it is a multiple of a lifting of the
barycenter of the vertices of $\cut^{\Box}(K_5)$ to $M$.
Alternatively, one checks the facet description of $\cut^{\Box}(K_5)$ in \cite[Chapter 30.6]{Deza-Laurent-10}
to deduce that statement. Hence, $\xbf\in \overline{M}_\ast$.

Assume that $\xbf\in M_\ast$. Then $\xbf\in M$ and one gets
\begin{equation}
\label{Eq-K5}
\xbf=
(\abf_{i_1}, 1)
+
(\abf_{i_2}, 1)
+
(\abf_{i_3}, 1)
+
(\abf_{i_4}, 1)
\text{ for some } i_j\in [15].
\end{equation}
Observe that the sum of the first 10 entries of $\xbf$ is $20$ and sum of the entries of
$\abf_i$'s are either $4$ or $6$.
The only possibility is then that exactly two of the involved $\abf_{i_j}$ have the property that the sum of their entries is equal 4 and for the other two it is $6$.
Without loss of generality we may assume that
$i_1,i_2 \in [5]$ and $i_3,i_4 \in [15]\setminus[5]$.
Next we see that $\abf_{i_1}+\abf_{i_2}$ has exactly three entries equal to zero.  Indeed, they correspond to the edges of the triangle with vertices $[5]\setminus\{i_1,i_2\}$.
If one compares $\abf_{i_3}$, $\abf_{i_4}$ at those
entries, then the possibilities are
$
(0,\ 0,\ 0),\
(1,\ 1,\ 0),\
(0,\ 1,\ 1),\
(1,\ 0,\ 1).
$
Note that these correspond to all cut vectors of a triangle. Anyhow, adding two of such vectors one cannot get $(2,\ 2,\ 2)$. Thus, (\ref{Eq-K5}) cannot exists, which concludes the proof.
\end{proof}

Recall that for two (graded) $\Kb$-algebras $A,B$
and an injective (homogeneous) homomorphism $\iota\colon A\to B$ one calls $A$ \emph{an algebra retract} of $B$,
if there exists a
(homogeneous) homomorphism $\pi\colon B\to A$
such that $\pi \circ \iota=\id_A$.
Note that given a face $F$ of a polytope $P$,
then it follows from \cite[Corollary 4.34]{Bruns-Gubeladze}
that $\Kb[F]$ is a graded retract of $\Kb[P]$, which
we call a \emph{face retract}. This is a useful concept, since, e.g., if in the retract situation above $B$ is seminormal or normal, then $A$ has also this property (see, e.g., \cite[Proposition 3.1]{Epstein-Nguyen}).

\begin{corollary}
\label{Cor-seminormal}
Let $G$ be a graph such that $\cut^{\Box}(G)$ is seminormal. Then $G$ is $K_5$-minor-free.
\end{corollary}
\begin{proof}
Assume that $G$ has a $K_5$-minor. Recall that this minor can be obtained by taking edge contractions only (plus possible deleting isolated vertices). Hence, the cut polytope corresponding to that $K_5$-minor
$\cut^{\Box}(K_5)$ is a face of $\cut^{\Box}(G)$
by \cite[Lemma 3.2(2)]{Sturmfels-Sullivant}. Since $\cut^{\Box}(G)$ is seminormal,
then so is $\cut^{\Box}(K_5)$ by the retract situation described above.
This is a contradiction to Theorem \ref{Theorem-K5}
and hence $G$ has to be $K_5$-minor-free.
\end{proof}

Note that if Conjecture \ref{Conj-StSu} is true, then normality and seminormality of cut polytopes are equivalent.
Remarkably, Laso\'{n} and Micha{\l}ek \cite{Lason-Michalek}
were able to prove the latter fact. Their proof is based on Corollary \ref{Cor-seminormal} and very interesting additional facts related to
the question in which situations cut polytopes are very ample.

\begin{theorem}[(\cite{Lason-Michalek})]
Let $G$ be a graph. Then the following statements are equivalent:
\begin{enumerate}
\item
$\cut^{\Box}(G)$ is normal;
\item
$\cut^{\Box}(G)$ is seminormal.
\end{enumerate}
In particular,
Theorem \ref{normality-Ohsugi} is true
by replacing the word normal with seminormal everywhere.
\end{theorem}

\section{Canonical modules of cut polytopes}
\label{Section-5}

Given a (graded) Cohen--Macaulay algebra
there exists its (graded) \emph{canonical module}
which captures many important properties of the algebra. For its algebraic definition
and the theory itself we refer to \cite{Bruns-Gubeladze, Bruns-Herzog}.
Let $M$ be a normal and thus a Cohen--Macaulay
affine monoid (i.e. its algebra $\Kb[M]$ has the property). The canonical module of $\Kb[M]$ has in this case a nice description due to Danilov and Stanley (see \cite[Theorem 6.7]{Stanley} or \cite[Theorem 6.31]{Bruns-Gubeladze}). Indeed, it is the ideal generated by $\inte(M)$ inside $\Kb[M]$.
The main goal of this section is to determine this module for certain normal cut polytopes.

A special class of Cohen--Macaulay algebras
are the Gorenstein ones, which are of great interest in commutative algebra. For a normal affine monoid $M$ it is well-known that $\Kb[M]$ is Gorenstein
if and only if
$$\inte(M) = \xbf + M \text{ for some } \xbf\in \inte(M)
$$
(see, e.g., \cite[Theorem 6.33]{Bruns-Gubeladze}).
Then we call $x$ also a \emph{generator} for $\inte(M)$.
Observe that Ohsugi classified in \cite[Theorem 3.4]{Ohsugi-Gorenstein} all normal Gorenstein cut polytopes:
\begin{theorem}
[(\cite{Ohsugi-Gorenstein})]
\label{Ohsugi-Gorenstein}
The cut polytope $\cut^{\Box}(G)$ of a graph $G$ is normal and Gorenstein if and only if  $G$ is $K_5$-minor-free and
$G$ satisfies one of the following conditions:
\begin{enumerate}
\item
$G$ is a bipartite graph without induced cycle of length greater or equal to $6$.
\item
$G$  is a bridgeless chordal graph.
\end{enumerate}
\end{theorem}
Using this result we can show:
\begin{theorem}
\label{Theorem-canonical-module}
Let $G$ be a graph such that $\cut^{\Box}(G)$ is normal and Gorenstein. Then the generator $\xbf$ of the canonical module $\inte(M_G)$ is either
\begin{enumerate}
\item
$\xbf=(1,\ 1,\ldots,\ 1,\ 2)$ if
$G$ is a bipartite graph without induced cycle of length greater or equal to $6$, or
\item
$\xbf=(2,\ 2,\ldots,\ 2,\ 4)$ if
$G$  is a bridgeless chordal graph.
\end{enumerate}
\end{theorem}
The above theorem is also stated in \cite[Remark 5.2]{Ohsugi-Gorenstein} but without a proof. For the convenience of the reader we show it in the following. We subdivide the proof of Theorem \ref{Theorem-canonical-module} and discuss at first the following two lemmas.
\begin{lemma}
\label{canonical_module_bipartite}
With the assumptions of Theorem \ref{Theorem-canonical-module}(i)
and $E(G)=\{e_1,\dots,e_m\}$
the generator $\xbf$ of the canonical module $\inte(M_G)$ is
\[
\xbf=(1,\ 1,\ldots,\ 1,\ 2)
\in \Zb^{m+1}.
\]
\end{lemma}
\begin{proof}
Since  $G$ is bipartite there is a partition of vertices $V$ of $G$ say
$V=V_1\cup V_2$ with $V_1\cap V_2=\emptyset$. All edges of $G$ are then of the form $v_1v_2$ with $v_1\in V_1$ and $v_2 \in V_2$.
Note that
$(1,\ 1,\ldots,\ 1,\ 1) \in \Zb^m$ is a cut vector
corresponding to $V_1$
and
$(0,\ 0,\ldots,\ 0,\ 0) \in \Zb^m$ is a cut vector corresponding to $\emptyset$, which yields
\[
\xbf
=
(1,\ 1,\ldots,\ 1,\ 2)
=
(1,\ 1,\ldots,\ 1,\ 1)
+
(0,\ 0,\ldots,\ 0,\ 1)
\in M_G  \subseteq \Zb^{m+1}.
\]
$\inte(M_G)$ contains no vectors of the form
$(\abf,1)\in \Zb^{m+1}$, which follows, e.g., from (\ref{InEq-facets1})
since $G$ does not contain triangles.
It remains to observe that $\xbf$ is an element of $\inte(M_G)$ to conclude that it is the generator of $\inte(M_G)$,
since we know by assumption that there can be at most one element in $\inte(M_G)$ with last coordinate $2$.

For this we recall again that $G$ contains no triangles and more generally it contains only induced cycles of even length, because it is bipartite. By assumption on the length of induced cycles, the only possible length for such a cycle is $4$.
It follows from (\ref{InEq-facets1}) and (\ref{InEq-facets2}) that $\xbf \in \inte(M_G)$ if and only if $0<x_e < 2$  where $e\in E(G)$ and
\begin{align*}
x_a <&\ x_b+x_c+x_d, & x_a+x_b+x_c<&\ x_d+4,
\\
x_b <&\ x_a+x_c+x_d, & x_a+x_b+x_d<&\ x_c+4,
\\
x_c <&\ x_a+x_b+x_d, & x_a+x_c+x_d<&\ x_b+4,
\\
x_d <&\ x_a+x_b+x_c, & x_b+x_c+x_d<&\ x_a+4
\end{align*}
for each induced cycle $C$ of $G$ of length $4$ with edges $\{a,b,c,d\}$ corresponding to some of the first $m$ coordinates in $\Zb^{m+1}$. This is easily verified for $\xbf$
and concludes the proof.
\end{proof}

\begin{lemma}
\label{canonical_ module_bridgeless_chordal}
With the assumptions of Theorem \ref{Theorem-canonical-module}(ii)
and $E(G)=\{e_1,\dots,e_m\}$
the generator $\xbf$ of the canonical module $\inte(M_G)$ is
\[
\xbf=(2,\ 2,\ldots,\ 2,\ 4)
\in \Zb^{m+1}.
\]
\end{lemma}
\begin{proof}
Set
$\xbf=(2,\ 2,\ldots,\ 2,\ 4)
\in \Zb^{m+1}$.
Since $G$ is bridgeless and chordal,
each edge belongs to a cycle, triangles are the only induced cycles of $G$ and thus each edge belongs to a triangle. At first observe that for any cycle $C$ of $G$
the element $\xbf$ satisfies trivially
$
\sum_{e\in C} x_e  \equiv 0 \modulo 2
$
and thus $\xbf\in\gp(M_G)$ by  (\ref{elt-in-group}).
By (\ref{InEq-facets1}) and (\ref{InEq-facets2})
we have  $\xbf \in\Qb_{\geq 0} M_G$ if and only if
\[
\sum_{f\in F}x_f-
\sum_{e\in E(C)\setminus F}
x_e
\leq 4(|F |-1),
\]
where $C$ ranges over all induced cycles of $G$ and $F\subseteq E(C)$ with $|F|$ odd.
As mentioned above, $C$ has to have length $3$ with edges $\{a,b,c\}$.

Hence, the inequalities
look like $x_a+x_b+x_c\leq 8$
and $x_a\leq x_b+x_c$ and others  by permuting the indices.
We see that $\xbf$ satisfies all desired inequalities. In particular,
\[
\xbf \in\Qb_{\geq 0} M_G\cap \gp(M_G)=M_G,
\]
because $M_G$ is normal by assumption. Since all inequalities are even strictly satisfied,
one obtains that
$\xbf\in \inte(M_G)$.
Assume that $\inte(M_G)$ contains a vector
$\ybf=(\abf,\ k)$ for some integer $1\leq k \leq 3$
and $\abf \in \Zb^{m}$.
It is easy to see that $k \geq 2$, e.g., by reducing this to the connected case and using the structure of cut polytopes.
Choose a triangle $C$ of $G$ with edges $\{a,b,c\}$.
Assume that $k=2$. Then by (\ref{InEq-facets1}) and (\ref{InEq-facets2}) we have
\[
y_a+y_b+y_c< 4,\hspace{4pt}
y_a< y_b+y_c,\hspace{4pt}
y_b< y_a+y_c,\hspace{4pt}
y_c< y_a+y_b.
\]
On the other hand by comparing the possible cut vectors
and their coordinates at $a,b,c$, we see that
$\abf$ has to be a sum of exactly two of the vectors
\[
(0,\ 0,\ 0), \hspace{4pt}
(1,\ 1,\ 0), \hspace{4pt}
(1,\ 0,\ 1), \hspace{4pt}
(0,\ 1,\ 1)
\]
and can verify that this is not possible. Hence, $k\geq 3$.
A similar argument with $4$ replaced by $6$ and considering
sums of three of the latter vectors also rules out the possibility $k=3$. This concludes the proof.
\end{proof}

\begin{proof}
[(Proof of Theorem \ref{Theorem-canonical-module})]
This follows from Theorem \ref{Ohsugi-Gorenstein},
Lemma \ref{canonical_module_bipartite} and
Lemma \ref{canonical_ module_bridgeless_chordal}.
\end{proof}
Observe that alternatively one can use also \cite[Theorem 6.33]{Bruns-Gubeladze} to show Theorem \ref{Theorem-canonical-module}.
\begin{example}
\
\begin{enumerate}
\item
Note that for a complete bipartite and $K_5$-minor-free graph every induced cycle is of length $4$. Hence, Theorem \ref{Ohsugi-Gorenstein}
yields that the cut polytope is normal and Gorenstein. By Theorem \ref{Theorem-canonical-module}(i) the generator of  $\inte(M_G)$ is
\[
(1,\ 1,\ldots,\ 1,\ 2).
\]
Note that complete bipartite graphs are
in particular Ferrers graphs (see, e.g., \cite[Example 5.10(3)]{Romer-Madani}) and one can see that the statements
from here also hold in this more general case.
\item
Let $G=G_1 \# G_2$ be a 0- or 1-sum
of two graphs $G_1$ and $G_2$, which are both $K_5$-minor-free and bipartite graphs without induced cycle of length $\geq 6$.
Then $G$ has the same properties as the $G_k$'s and thus the generator $\xbf$ of the canonical module $\inte(M_G)$ is
by Theorem \ref{Theorem-canonical-module}(i)
\[
\xbf=(1,\ 1,\ldots,\ 1,\ 2).
\]
\item
Let $G=G_1 \# G_2$ be a $i$-sum for $i=0,1,2,3$
of two graphs $G_1$ and $G_2$, which are both $K_5$-minor-free and bridgeless chordal graphs.
Then $G$ has the same properties as the $G_k$'s and thus
the generator $\xbf$ of the canonical module $\inte(M_G)$ is
by Theorem \ref{Theorem-canonical-module}(ii)
\[
\xbf=(2,\ 2,\ldots,\ 2,\ 4)
\]
\end{enumerate}
\end{example}
For certain clique-sums one can determine $\inte(M_G)$
for a given graph $G$ explicitly like in the previous example.
Another case of this type is:
\begin{lemma}
\label{zero-sum}
Let $G_1$ and $G_2$ be two $K_5$-minor-free graphs, $G=G_1 \#_0 G_2$ and $W_1, W_2 \subseteq \mathbb{N}$ be two finite sets. Suppose that a system of generators of
$\inte(M_{G_1})$ has last coordinates $\beta_1$ with
$\beta_1 \in W_1$
and a system of generators of
$\inte(M_{G_2})$ has last coordinates $\beta_2$ with
$\beta_2 \in W_2$.
Then $\inte(M_{G})$ has a system of generators with elements of the form
$(\xbf,\ \ybf,\ \beta)$ with
\[
\beta \in W_1\cup W_2,\
\beta\geq \max(\min W_1,\min W_2),\
(\xbf,\ \beta) \in \inte(M_{G_1})
\text{ and }
(\ybf,\ \beta) \in \inte(M_{G_2}),
\]
where
for any choice $\beta\in W_1\cap W_2$,
any generator
$(\xbf,\ \beta) \in \inte(M_{G_1})$
and any generator
$(\ybf,\ \beta) \in \inte(M_{G_2})$
a vector
$(\xbf,\ \ybf,\ \beta)$ is part of that system.
\end{lemma}
\begin{proof}
Let $G=G_1\#_0 G_2$ be a 0-sum of two graphs $G_1$ and $G_2$.
The cut vectors of $G$ are exactly of the form
$(\delta_{A_1},\ \delta_{A_2})$
for $A_1\subseteq V(G_1)$ and $A_2\subseteq V(G_2)$.
Thus,
\[
M_G
=
\{
(\xbf,\ \ybf,\ \alpha):
(\xbf,\ \alpha)\in M_{G_1}
\text{ and }
(\ybf,\ \alpha)\in M_{G_2}
\}.
\]
Suppose that
$
(\xbf,\ \ybf,\ \alpha)\in \inte(M_G).
$
Using Proposition \ref{Prop-facets-K5-minor-free}
one sees that
$(\xbf,\ \alpha)\in \inte(M_{G_1})$
and
$(\ybf,\ \alpha)\in \inte(M_{G_2})$.
There exists a generator
$(\Tilde{\xbf},\ \beta_1)$ of $\inte(M_{G_1})$
and a generator
$(\Tilde{\ybf},\ \beta_2)$ of $\inte(M_{G_2})$
such that
\[
(\vbf,\ \alpha-\beta_1)
=
(\xbf,\ \alpha)
-(\Tilde{\xbf},\ \beta_1) \in M_{G_1}
\text{ and }
(\wbf,\ \alpha-\beta_2)
=(\ybf,\ \alpha)
-(\Tilde{\ybf},\ \beta_2) \in M_{G_2},
\]
where $\beta_1\in W_1$ and $\beta_2\in W_2$. Assume that $\beta_1=\beta_2=\beta$. Then
\[
(\xbf,\ \ybf,\ \alpha)
=
(\Tilde{\xbf},\ \Tilde{\ybf},\ \beta)
+
(\vbf,\ \wbf,\ \alpha-\beta)
\in
(\Tilde{\xbf},\ \Tilde{\ybf},\ \beta)
+M_G
\]
with
$
(\Tilde{\xbf},\ \Tilde{\ybf},\ \beta)
\in \inte(M_G)$
and
$(\vbf,\ \wbf,\ \alpha-\beta)\in M_G$.

Next suppose that $\beta_1<\beta_2\leq \alpha$.
Note that $M_{G_1}$ is generated by certain elements
of the form
$(\abf_1,\ 1),\dots,(\abf_n,\ 1)$.
There exists a presentation
$
(\vbf,\ \alpha-\beta_1)
=
\sum_{k=1}^{\alpha-\beta_1}
(\abf_{i_k},\ 1)
$
with $i_k \in [n]$. Set
\[
(\Tilde{\zbf},\ \beta_2)
=
(\Tilde{\xbf},\ \beta_1)
+
\sum_{k=1}^{\beta_2-\beta_1}
(\abf_{i_k},\ 1)
\in \inte(M_{G_1}).
\]
Then
$
(\xbf,\ \alpha)
\in (\Tilde{\zbf},\ \beta_2)
+M_{G_1}
$
and thus, as above,
$
(\xbf,\ \ybf,\ \alpha)
\in
(\Tilde{\zbf},\ \Tilde{\ybf},\ \beta_2)
+M_G.
$
The case $\beta_1>\beta_2$ is treated in the same way
and this concludes the proof.
\end{proof}
This rather technical lemma can for example be used in the following situation:

\begin{example}
\label{canonical-module-chordal}
Assume that $G$ is a $K_5$-minor-free and chordal graph.
The property chordal is by  \cite[Proposition 5.5.1]{Diestel} equivalent to the fact that $G$ is a clique-sum of complete graphs.
Using the $K_5$-minor-freeness this means
that $G$ has to be a clique-sum of certain $K_2$, $K_3$ and $K_4$.
Reordering this we see that $G$ is a $0$-sum
of certain $K_5$-minor-free bridgeless chordal graphs (where no $K_2$ is allowed to use) and some trees. Thus,
\[
G=G_1\#_0 T_1\#_0 G_2\#_0T_2\#_0\dots G_n\#_0T_n,
\]
where $G_i$ are $K_5$-minor-free bridgeless chordal graphs and  $T_i$ are trees.

For a tree $T$ it follows from (\ref{elt-in-group}),
(\ref{InEq-facets1}) and (\ref{InEq-facets2})
that its elements $(\xbf,\alpha)\in \inte(M_T)$
have to satisfy $1\leq x_e\leq \alpha-1$ for $e\in E(T)$ and $x_e\in \Zb$. As a (not minimal) system of generators
we can choose
$(\xbf,4)$ with $1\leq x_e\leq 3$, $x_e\in \Zb$
together with all elements in $\inte(M_T)$ with $2\leq \alpha \leq 3$. Let $\twobf_k=(2,\ldots,2)\in \Zb^{E(G_k)}$ for $k\in [n]$.
The observation for trees, Lemma \ref{canonical_ module_bridgeless_chordal} and Lemma \ref{zero-sum} imply that $\inte(M_G)$ has the following system of generators
\[
\biggl\{
(\twobf_1,\ \xbf_1,\ldots,\ \twobf_n,\ \xbf_n,\ 4)
: 1\leq (\xbf_k)_e \leq 3,\ (\xbf_k)_e \in \Zb \text{ for } e\in E(T_k)
\biggr\}.
\]
\end{example}
\begin{corollary}
\label{zero-sum-corollary}
 Let $G_1$ be a graph which is both $K_5$-minor-free and bridgeless chordal and $G_2$ be a graph which is both $K_5$-minor-free and bipartite without induced cycle of length $\geq 6$. Let $G=G_1\#_0 G_2$.
 Then $\inte(M_{G})$ has a system of generators of the form
\[
\bigl\{
(\twobf,\ \ybf,\ 4):\
(\twobf,\ 4) \in \inte(M_{G_1})\
\text{ and }\
(\ybf,\ 4) \in \inte(M_{G_2})
\bigr\}.
\]
\end{corollary}
\begin{proof}
The corollary follows from Theorem \ref{Theorem-canonical-module}
and Lemma \ref{zero-sum}.
\end{proof}
\begin{lemma}
\label{one-sum}
 Let $G_1$ be a graph which is both $K_5$-minor-free and bridgeless chordal and $G_2$ be a graph which is both $K_5$-minor-free and bipartite without induced cycle of length $\geq 6$. Let $G=G_1\#_1 G_2$.
 Then $\inte(M_{G})$ has a system of generators with elements of the form
\[
(2,\ \twobf,\  \ybf,\ 4)
\text{ such that }\
(2,\ \twobf,\  4) \in \inte(M_{G_1})\
\text{ and }\
(2,\ \ybf,\ 4) \in \inte(M_{G_2}).
\]
\end{lemma}
\begin{proof}
Assume that with respect to a suitable ordering
the common edge $e$ of $G_1$ and $G_2$,
at which the 1-sum is performed, is the first edge of $E(G)$,
then the remaining edges from $E(G_1)$ follow
and finally we see the edges from $E(G_2)$.
Observe that
\[
M_G
=
\{
(\gamma,\ \xbf,\ \ybf,\ \alpha):
(\gamma,\ \xbf,\ \alpha)\in M_{G_1}
\text{ and }
(\gamma,\ \ybf,\ \alpha)\in M_{G_2}
\}.
\]
Suppose that
$
(\gamma,\ \xbf,\  \ybf,\ \alpha)\in \inte(M_G).
$
It follows from Proposition \ref{Prop-facets-K5-minor-free}
that
$(\gamma,\ \xbf,\  \alpha)\in \inte(M_{G_1})$
and
$(\gamma, \ \ybf,\ \alpha)\in \inte(M_{G_2})$. Hence,  using Theorem \ref{Theorem-canonical-module},
\[
(\gamma,\ \xbf,\  \alpha)
=
(2,\ \twobf,\  4)
+
(\gamma-2,\ \vbf,\  \alpha-4),
\text{ where }
(\gamma-2,\ \vbf,\ \alpha-4)
 \in M_{G_1}.
 \]
In particular, one obtains $\gamma-2 \leq \alpha-4$ by (\ref{InEq-facets2}) and thus
$\gamma \leq  \alpha-2$. We see also that
 \[
(\gamma,\ \ybf,\ \alpha)
=
(1,\ \onebf,\  2)
+(\gamma-1,\ \wbf,\ \alpha-2),
\text{ where }
(\gamma-1,\ \wbf,\ \alpha-2)
\in M_{G_2}.
\]
Recall that the affine monoid $M_{G_1}$ is generated by elements of the form
\[
(\delta^1_1(e),\ \abf_1,\  1), \ldots, (\delta^1_{n_1}(e),\ \abf_{n_1},\  1),
\]
and the affine monoid $M_{G_2}$ is generated by elements
of the form
\[
(\delta^2_1(e),\ \bbf_1,\ 1), \ldots, (\delta^2_{n_2}(e),\ \bbf_{n_2}, \ 1).
\]
There exists a presentation,
 $(\gamma-1,\ \wbf,\ \alpha-2)=\sum_{k=1}^{\alpha-2}
(\delta^2_{i_k}(e),\ \bbf_{i_k},\ 1)
\text{ with }
i_k \in [n]$.
As noted above $\gamma\leq \alpha-2$. Hence, $ \gamma-1< \alpha-2$
and there must exist  $j$  and $j'$ such that $\delta^2_{i_j}(e)=0$ and
$\delta^2_{i_{j'}}(e)=1$.
Set
\[
(2,\ \Tilde{\zbf},\ 4)
=(1,\ \onebf,\ 2)
+(0,\ \bbf_{i_j},\ 1)
+(1,\  \bbf_{i_{j'}},\ 1)
\in \inte(M_{G_2}).
\]
Then
$
(\gamma,\ \ybf, \ \alpha)
\in (2,\ \Tilde{\zbf},\ 4)
+M_{G_2}
$
and thus,
$
(\gamma,\ \xbf,\  \ybf,\ \alpha)
\in
(2,\ \twobf,\   \Tilde{\zbf},\ 4)
+M_G.
$
\end{proof}
\begin{example}
Let $G$ be a ring graph without induced cycle of length $\geq 5$. Since it is a $0$- or $1$- sum of cycles and trees, $G$ can be obtained from a sequence of $0$- or $1$- sum of triangles, squares, and trees.
Recall that $\inte(M_C)$ is generated by
$(2,\ 2,\ 2,\ 4)$ for a triangle $C$ and $\inte(M_D)$ is generated by  $(1,\ 1,\ 1,\ 1,\ 2)$ for a square $D$. Then by Corollary \ref{zero-sum-corollary} and Lemma \ref{one-sum},
the canonical module $\inte(M_G)$ has generators whose last coordinates are $4$.

In Example \ref{canonical-module-chordal}, for a tree $T$, we have discussed
the structure of elements $(\xbf,4)\in \inte(M_T)$. For a square $D$, by analyzing its cut vectors, it can be seen that
$(\xbf,\ 4) \in \inte(M_D)$ if and only if for each $e\in E(D)$ we have $1\leq x_e\leq 3$ with $x_e\in \Zb$, and either all entries of $\xbf$ are equal or they are pairwise equal. Hence, again using Corollary \ref{zero-sum-corollary} and Lemma \ref{one-sum},  one gets a complete description of a set of generators of $\inte(M_G)$.
\end{example}
\section{Castelnuovo--Mumford regularity of cut polytopes}
\label{Section-6}

In this section we study the Castelnuovo--Mumford regularity of cut polytopes,
i.e.  the Castelnuovo--Mumford regularity of their cut algebras.
Let us recall the definition. For a polynomial ring
$R=\Kb[x_1,\ldots,x_n]$ and a finitely generated graded
$R$-module $M\neq 0$ its  \emph{graded Betti numbers} are
\[
\beta_{i,j}^R(M):=\dim_{\Kb} \Tor_i^R(\Kb,M)_j \text{ for } i=0,\dots,n \text{ and } j\in \Zb.
\]
The \emph{Castelnuovo--Mumford regularity}  of $M$ is defined as
\[
\reg_R M =\max\{j-i:\beta_{i,j}^R(M)\neq 0\}.
\]
For a graph $G$ with cut algebra $\Kb[G]$  and cut ideal $I_G$
we determine their regularities always with respect to the polynomial ring
$S_G$ as defined in Section \ref{Section-2}.
To simplify the notation we omit $S_G$ as an index and simply write
$\reg \Kb[G]$ and $\reg I_G$ for the regularities. Observe that by
\cite[Proposition 3.2]{Romer-Madani} we have $(I_G)_1=0$ if and only if $G$ is connected.
So in this case $\reg I_G \geq 2$ or $I_G=0$. Recall that $\dim \Kb[G]=|E(G)|+1$.

In the following we need the following useful fact.
Let $T=S/J$ be a standard graded algebra where $S$ is a polynomial ring and $J$ is a graded ideal.
If $J\neq 0$, then it follows from the definition that
\[
\reg_S T =\reg J -1.
\]
If $T$ is Cohen--Macaulay with graded canonical module $\omega_T$
and  (Krull-)dimension $d$, then (e.g., by \cite[Equation (6.6)]{Bruns-Gubeladze})
\begin{equation}
\label{Eq-reg}
\reg_S T = d- \min\{i\in \Zb:\ (\omega_T)_i \neq 0 \}.
\end{equation}

At first we discuss a general lower bound for the regularity in the normal case.
\begin{proposition}
\label{lower-bound}
Let $G=(V,E)$ be a graph such that $\cut^{\Box}(G)$ is normal. Then
\[
\reg \Kb[G]\geq |E|-3.
\]
\end{proposition}
\begin{proof}
Set
\[
\vbf=(\xbf,\ \alpha)
=
(2,\ 2,\ \ldots,\ 2,\ 4)\in \Zb^{|E|+1}.
\]
Note that $\vbf \in \gp(M_G)$ by (\ref{elt-in-group}).
Trivially $\vbf$ satisfies strictly the inequalities in (\ref{InEq-facets1}). For those in (\ref{InEq-facets2}) let us consider
an induced cycle $C$ of $G$ of length $n\geq 3$ and let $F\subseteq E(C)$ be an odd subset.
For $\vbf$ we compute
\[
\sum_{f\in F}x_f-
\sum_{e\in E(C)\setminus F}
x_e
=
2
|F|
-
2(n-|F|)
<
4(|F |-1).
\]
Thus, also the inequalities in (\ref{InEq-facets2}) are strictly satisfied
and this implies
\[
\vbf\in \Qb_{\geq 0}(M_G) \cap \inte \bigl(\cone(M_G)\bigr).
\]
By assumption $\cut^{\Box}(G)$ is normal.
Hence, $\Qb_{\geq 0}(M_G)\cap \gp(M_G)=M_G$ and we get
\[
\vbf\in M_G \cap \inte \bigl(\cone(M_G)\bigr)=\inte(M_G).
\]
Let $\omega$ be the canonical module of $\Kb[G]$.
Then $\vbf\in \inte(M_G)$ implies $\omega_4\neq 0$.
Using (\ref{Eq-reg}) this yields
\[
\reg \Kb(G)
=
\dim \Kb[G]- \min\{i\in \Zb:\ \omega_i \neq 0 \}
\geq
|E|+1-4=|E|-3.
\]
\end{proof}

For our next main result we observe before:

\begin{lemma}
\label{regularity-bipartite}
Let $G=(V,E)$ be a bipartite graph with $|E|\geq 1$ such that $\cut^{\Box}(G)$ is normal.
Then
\[
\reg \Kb[G]\geq |E|-1.
\]

\end{lemma}
\begin{proof}
Let $\omega$ be the canonical module of $\Kb[G]$.
As in the beginning of the proof of Lemma \ref{canonical_module_bipartite}
we see that $\xbf=(1,\ \dots,\ 1,\ 2)\in M_G$.
We claim that $\xbf \in \inte(M_G)$.
Let $C$ be an induced cycles of $G$ of even length $n\geq 4$
and $F\subseteq E(C)$ with $|F|$ odd.
Note that $|F|-(n-|F|)< 2(|F|-1)$ and thus
\[
\sum_{f\in F} x_f -\sum_{e\in E(C)} x_e < 2(|F|-1).
\]
Moreover,  $0<x_e < 2$  where $e\in E(G)$.
Hence, it follows from (\ref{InEq-facets1}) and (\ref{InEq-facets2})
that $\xbf \in \inte(M_G)$ and so $\omega_2\neq 0$.
Using Equation (\ref{Eq-reg}) this yields
\[
\reg \Kb[G] \geq \dim\Kb[G]-2 =|E|+1-2=|E|-1,
\]
which concludes the proof.
\end{proof}

We can also provide upper bounds for normal cut polytopes,
which often are equalities.

\begin{theorem}
\label{regularity-main}
Let $G=(V,E)$ be a graph such that $\cut^{\Box}(G)$ is normal. Then
\[
\reg \Kb[G]
\begin{cases}
=|E|-1&\text{if any induced cycles of $G$ is of even length (i.e., it is bipartite)},\\
\leq |E|-2&\text{if $G$ has no triangles and has an induced cycle of odd length},\\
=|E|-3&\text{if $G$ contains a triangle}.
\end{cases}
\]
\end{theorem}
\begin{proof}
Let $\vbf=(\xbf,\ \alpha)\in \inte(M_G)$
with $\alpha\in \Nb$ chosen in a way that (using (\ref{Eq-reg}))
\[
\reg \Kb(G)
=
|E|+1-\alpha.
\]

\emph{Case 1}: Assume at first that any induced cycles of $G$ (if existing)
is of even length and in particular, $G$ contains no triangles.
It follows from (\ref{InEq-facets1}) that in this case $\alpha \geq 2$, since
the facet defining inequalities $0\leq x_e\leq 1$ have no chance to be strictly satisfied by integers.
This implies already $\reg \Kb[G]\leq |E|-1$. The equality follows from this and
Lemma \ref{regularity-bipartite}.

\emph{Case 2}: Next we consider the case that $G$ contains
no triangles, but an induced cycle $C$ of odd length.
As in Case 1 we see immediately $\alpha \geq 2$.
In particular, for $e\in E(G)$ and the coordinate $x_e$ of $\vbf$
we get by (\ref{InEq-facets1}) that
\[
1\leq x_e \leq \alpha-1.
\]
Assume that $\alpha=2$. Then $x_e=1$ for $e\in E(G)$.
But for $e\in E(C)$ and the coordinates $x_e$ of $\vbf$, the congruences in
(\ref{elt-in-group}) imply the contradiction
$
|E(C)|=\sum_{e\in C} x_e
\equiv 0
\modulo 2.
$
Hence, $\alpha \geq 3$ and
$
\reg \Kb[G]\leq |E|-2.
$

\emph{Case 3}: Finally, we consider the case that $G$ contains (at least) one triangle $C$
with edges $\{e, f, g\}$.
Then $\vbf=(\xbf,\ \alpha)\in \inte(M_G)$. Using
(\ref{elt-in-group}), (\ref{InEq-facets1}) and (\ref{InEq-facets2})
applied to the cycle $C$ yield that $(x_e,\ x_f,\ x_g,\ \alpha)\in \inte(M_C)$.
Then by Theorem \ref{Theorem-canonical-module}(ii) we get
that $\alpha \geq 4$. Thus, $\reg \Kb[G]\leq |E|-3$.
This together with the lower bound of Proposition \ref{lower-bound}
yields the equality in this case.
\end{proof}

\begin{corollary}
\label{regularity-even}
Let $G=(V,E)$ be a cycle of even length. Then
\[
\reg \Kb[G]=|E|-1.
\]
\end{corollary}
\begin{proof}
The proof follows from Corollary \ref{ring-graph} and from the first case of Theorem \ref{regularity-main}.
\end{proof}

Having Theorem \ref{regularity-main} in mind, it is an interesting
question to find a class of graphs where where always equality occurs instead of inequalities in the second case.
We will see that ring graphs have this property.  The study of the regularity of cut algebras for ring graphs was initiated in \cite[Proposition 4.4 and Corollary 6.5]{NP}. Using a different approach
we give in Theorem~\ref{regularity-ring-graph} a complete description of the regularity in that case. To prepare the proof of this result,  we discuss at first a lemma.

\begin{lemma}
\label{regularity-odd}
Let $G=(V,E)$ be a cycle of odd length $2n+1>3$ for $n\in \Nb$. Then
\[
\reg \Kb[G]=|E|-2.
\]
\end{lemma}
\begin{proof}
Note that  $\cut^{\Box}(G)$ is normal by Corollary \ref{ring-graph}.
Let $\omega$ be its canonical module. By the second case of Theorem \ref{regularity-main} we have $\reg \Kb[G]\leq |E|-2$.

Let $V=\{v_1, \cdots,v_{2n+1}\}$  and $E=\{e_1,\cdots,e_{2n+1}\}$,
where $e_i=v_iv_{i+1}$ for $1\leq i< {2n+1}$ and
$e_{2n+1}=v_{2n+1}v_1$.

Consider the cut vectors
$\abf$ corresponding to $\{v_1,v_3, \ldots, v_{2n+1}\} \subseteq V$
and
$\bbf$ corresponding to $\{v_1\} \subseteq V$.
Thus, $\abf_{e_i}=1$ for $i\neq 2n+1$ and $\abf_{e_{2n+1}}=0$
as well as
$\bbf_{e_i}=1$ for $i=1,2n+1$ and
$\bbf_{e_i}=0$ for $1<i<2n+1$.
It follows that
\[
\vbf=
(2,\ 1,\ 1,\ \ldots,\ 1,3)
=
(\abf,1)
+
(\bbf,1)
+
(\zerobf,1)
\in M_G.
\]
Let $F\subseteq E$ be an odd subset.
If $e_1\in F$, then
\[
|F|+1< 3(|F|-1)+(2n+1-|F|),
\]
since $2n+1+|F|-4>0$.
If $e_1\not\in F$, then similarly
\[
|F|<3(|F|-1)+(2n+1-|F|+1).
\]
Using (\ref{InEq-facets1}) and (\ref{InEq-facets2})
one gets that $\vbf \in \inte(M_G)$ and then $\omega_3\neq 0$.
Hence, using Equation (\ref{Eq-reg}) this yields
\[
\reg \Kb[G] \geq \dim\Kb[G]-3 =|E|+1-3=|E|-2.
\]
Thus, $\reg \Kb[G]=|E|-2$.
\end{proof}

For the proof of the next theorem we observe the following.
A ring graph $G$ may have isolated vertices, which are not relevant
for us since they have no affect on cut polytopes. So
we always may assume without loss of generality that $G$ is connected for such a graph. $G$ can be a cycle or a tree. If this is not the case, then
there exists always a decomposition
of one of the following types
(see, e.g., \cite[Example 5.10(4)]{Romer-Madani}):
\begin{enumerate}
\item
$G=\Tilde{G}\times_k T$ for $k\in \{0,1\}$
where $\Tilde{G}$ is a ring graph and $T$ is a tree.
One even can reduce the case $k=1$ to (at most)
two decompositions of type $\Tilde{G}\times_0 T$.
\item
$G=\Tilde{G}\times_k C$ for $k\in \{0,1\}$
where $\Tilde{G}$ is a ring graph
and $C$ is a cycle.
\end{enumerate}
Note also that induced cycles of $G$ are the ones of $\Tilde{G}$
in both cases, or in (ii) additionally the cycle $C$.
With this preparation, we are ready to prove:
\begin{theorem}
\label{regularity-ring-graph}
Let $G=(V,E)$ be a ring graph with $|E|\geq 1$.
Then:
\[
\reg \Kb[G]
=
\begin{cases}
|E|-1&\text{if any induced cycles of $G$ is of even length},\\
|E|-2&\text{if $G$ has no triangles and has an induced cycle of odd length},\\
|E|-3&\text{if $G$ contains a triangle}.
\end{cases}
\]
\end{theorem}
\begin{proof}
The polytope  $\cut^{\Box}(G)$ is normal by Corollary \ref{ring-graph}.
Let $\omega$ be the associated canonical module
generated as an ideal by $\inte(M_G)$ which trivially satisfies $\omega_0=0$.
Without loss of generality we may assume that $G$ is connected.

\emph{Cases 1, 3}:
These cases follow directly from Theorem \ref{regularity-main}.

\emph{Case 2}:
Next we consider the case of a ring graph $G$, which does not
contain triangles, but has induced cycles of odd length.
It follows from the second case of Theorem \ref{regularity-main} that  $\reg \Kb[G]\leq |E|-2$.  We claim that there exists an $\xbf\in \Zb^{E(G)}$ such that
\[
(\xbf,\ 3)\in \inte(M_G).
\]
Note that by (\ref{InEq-facets1}) necessarily $1\leq x_e \leq 2$ for any $e\in E(G)$. Then (\ref{Eq-reg}) yields that $\reg \Kb[G]\geq |E|-2$ and it remains to prove the claim for Case 2.

If $G$ is equal to just such an odd cycle, then Lemma \ref{regularity-odd}
and its proof yield the assertion of the theorem as well as the claim.

We show the claim by an induction on the number of $0$- and $1$-sums
at most needed to construct $G$ as a ring graph using trees and cycles as discussed above, where we already verified the base case and consider for trees only $0$-sums.

Assume that $G=G_1\#_1G_2$ is a $1$-sum of two ring graphs $G_1$, $G_2$
which both do not contain triangles and $G_2=C$ is a cycle.
Let the edge $e=vw$, at which the 1-sum is performed, be the first edge of $E(G)$,
then we assume that the remaining edges
from $E(G_1)$ follow and finally we see the edges
from $E(G_2)$ with respect to a suitable ordering.

If $G_1$ contains (if existing) only cycles of even length,
then $G_2$ has to be an odd cycle.
In $\inte(M_{G_1})$ there exists
by the proof of Lemma \ref{regularity-bipartite}
an element $(1,\ 1, \ldots,\ 1,\ 2)$ and then also
\[
(1,\ \vbf_1,\ 3)
=
(1,\ 1, \ldots,\ 1,\ 3)
=
(1,\ 1, \ldots,\ 1,\ 2)
+
(0,\ 0, \ldots,\ 0,\ 1)\in \inte(M_{G_1}),
\]
since $(0,\ 0, \ldots,\ 0,\ 1)\in M_{G_1}$
as the vector induced by the cut vector of the empty set.
It follows from Lemma \ref{regularity-odd} that $\inte(M_{G_2})$  contains
$
(1,\ \vbf_2,\ 3)
=
(1,\ 2,\ 1,\dots,\ 1,\ 3),
$
since for such vectors the $2$ might be at any position (except the last one).
One gets that
\[
(1,\ \vbf_1,\ \vbf_2,\  3)
\in \inte(M_G),
\]
which follows from the fact
that induced cycles of $G$ are those from $G_1,G_2$,
the monoid $M_G$ is normal (thus $\Qb_{\geq 0} M_G\cap \gp(M_G)=M_G$),
and then applying all this using
(\ref{elt-in-group}), (\ref{InEq-facets1}) and (\ref{InEq-facets2}).

Next we assume that $G_1$ has induced cycles of odd length
and $G_2$ might be an arbitrary cycle not equal to a triangle.
If by induction hypothesis
\[
(1,\ \vbf_1,\ 3)\in \inte(M_{G_1}),
\]
then one uses facts in the same manner as before
by observing that
$(1,\ 1, \ldots,\ 1,\ 3)\in \inte(M_{G_2})$
if $G_2$ is an even cycle, or
$(1,\ 2,\ 1, \ldots,\ 1,\ 3)\in \inte(M_{G_2})$
if $G_2$ is an odd cycle, to conclude
that
$
(1,\ldots ,\  3)
\in \inte(M_G).
$

It remains to consider the subcase
\[
(2,\ \vbf_1,\ 3)\in \inte(M_{G_1}).
\]
If $G_2$ is an odd cycle, then
Lemma \ref{regularity-odd} implies that $\inte(M_{G_2})$  contains
\[
(2,\ \vbf_2,\ 3)
=
(2,\ 1,\dots,\ 1,\ 3)\in \inte(M_{G_2}).
\]
Since all monoids are normal,
that induced cycles of $G$ are those from $G_1,G_2$,
and then applying again
(\ref{elt-in-group}), (\ref{InEq-facets1}) and (\ref{InEq-facets2})
yields
\[
(2,\ \vbf_1,\ \vbf_2,\ 3)\in \inte(M_G).
\]
Finally, assume that $G_2$ is an even cycle and thus
\[
(1,\ 1,\ 1,\dots,\ 1,\ 2)\in \inte(M_{G_2})
\]
by the proof of Lemma \ref{regularity-bipartite}.
Since $G_2$ is bipartite, one can see
analogously to the proof of Lemma \ref{canonical_module_bipartite}
that
\[
(1,\ 1,\ 1,\dots,\ 1,\ 1)\in M_{G_2}
\]
and thus
\[
(2,\ \vbf_2,\ 3)
=
(1,\ 1,\ 1,\dots,\ 1,\ 2)
+
(1,\ 1,\ 1,\dots,\ 1,\ 1)
\in  \inte(M_{G_2}).
\]
In this case
we get by our ``standard'' argument that
\[
(2,\ \vbf_1,\ \vbf_2,\ 3)\in \inte(M_G).
\]
Similarly, one proves the case that
$G=G_1\#_0G_2$ is a $0$-sum using graphs $G_1,G_2$ as above.
Here the case distinction is much simpler, since one does not have
to make sure that one particular chosen coordinate is
equal for vectors produced in $\inte(M_{G_1})$ and $\inte(M_{G_2})$.
\end{proof}

Before we discuss our next main result,
we need to determine the regularity for one more special case. In the following
$P_n$ denotes always a path of (edge-)length $n$, i.e.\ it has exactly $n$ edges.

\begin{lemma}
\label{regularity-example}
We have $\reg \Kb[C_5\#_{P_2}C_4]=5$.

\end{lemma}
\begin{proof}
Let $G=C_5\#_{P_2}C_4$ and $E=E(G)$ as well as $V=V(G)$. Since $G$ is $K_5\setminus e$-minor-free, it follows from Proposition~\ref{known-cases-conjecture} (i)
that $\cut^{\Box}(G)$ is normal.
Let $\omega$ be its canonical module. By the second case of Theorem \ref{regularity-main} we know already that
$
\reg \Kb[G]\leq
7-2=5.
$
Choose an ordering of the vertices and edges.
Let $V=\{v_1, \ldots,v_{6}\}$  and $E=\{e_1,\ldots,e_{7}\}$,
where  $e_i=v_iv_{i+1}$ for $1\leq i< {5}$,
$e_{5}=v_{5}v_1$, $e_{6}=v_1v_6$ and $e_{7}=v_3v_6$.

That is, the cycle $C_5$ with vertices $\{v_1, \cdots,v_{5}\}$ is  attached via the $P_2$-sum performed
at the edges $\{e_1,e_2\}$
with $C_4$ with vertices $\{v_1,v_2,v_3,v_6\}$.

Consider the cut vectors
$\abf$ corresponding to $\{v_1,v_3, v_{5}\} \subseteq V$
and
$\bbf$ corresponding to $\{v_5\} \subseteq V$.
Thus, $\abf_{e_i}=1$ for $i\neq 5$ and $\abf_{e_{5}}=0$
as well as
$\bbf_{e_i}=1$ for $i=4,5$ and
$\bbf_{e_i}=0$ for any other $i$.
It follows that
\[
\vbf=
(1,\ 1,\ 1,\ 2,\ 1,\ 1,\ 1,\ 3)
=
(\abf,1)
+
(\bbf,1)
+
(\zerobf,1)
\in M_G.
\]
Using the proofs of
Lemma \ref{regularity-bipartite}
and
Lemma \ref{regularity-odd}
as well as
that $C_4,C_5$ are the only induced cycles of $G$
one gets that $\vbf \in \inte(M_G)$ and then $\omega_3\neq 0$.
Hence, Equation (\ref{Eq-reg}) yields
\[
\reg \Kb[G] \geq \dim\Kb[G]-3 =7+1-3=5,
\]
and thus $\reg \Kb[G]=5$.
\end{proof}
The strategy of the previous proof
can also be used to show that $\reg \Kb[C_{2n+1}\#_{P_2}C_{2k}]=2n+2k-3$ for $n\geq 1$ and $k\geq 2$.
Note that cut polytopes with small regularity are normal.
\begin{lemma}
\label{small-regularity-normal}
 Let $G$ be a graph such that $\reg \Kb[G]\leq 4$. Then $\cut^{\Box}(G)$  is normal.
\end{lemma}

\begin{proof}
For this proof we use the fact that $I_{K_5}$ has a minimal generator in degree 6
(see, e.g., \cite[Table 1]{Sturmfels-Sullivant}),
which implies that $\reg I_{K_5}\geq 6$ and thus $\reg \Kb[K_5]\geq 5$.
Assume that $G$ has $K_5$ as a minor. Recall that this
minor can be obtained by edge contractions only
(plus possible deletions of isolated vertices, etc.)
and that $\Kb[K_5]$ is an algebra retract of $\Kb[G]$
(see \cite[Corollary 4.6]{Romer-Madani}).
Then, e.g., \cite[Proposition 2.2]{Romer-Madani} implies that
$$
\reg \Kb[G]\geq \reg \Kb[K_5]\geq 5.
$$
This is a contradiction. Hence, $G$ is $K_5$-minor-free.

By Theorem \ref{regularity-main} or Theorem \ref{Theorem-canonical-module}(ii)
we know that $\reg \Kb[K_5\setminus e]=6$.

Assume that $H=K_5\setminus e$ is a minor of $G$.
Then we can get it from $G$ by certain edge deletions and edge contractions.
Observe that the deletions do not decrease the number of vertices, but
contractions have this property  and also that
operations deletions/contractions commute with each other.

Just by using only the edge contractions with respect to $G$ (plus possible deletions of isolated vertices, etc.)
yields a graph $\Tilde{H}$ on five vertices which contains $H$ as a subgraph.
The only possibilities of such $\Tilde{H}$ are $H$ and $K_5$.
Since $G$ is $K_5$-minor-free, we obtain that $\Tilde{H}=K_5\setminus e$.
Thus, the minor $K_5\setminus e$ can be obtained from $G$ using only
edge contractions. But then \cite[Corollary 4.6]{Romer-Madani} implies again
the contradiction
$$
\reg \Kb[G]\geq \reg \Kb[K_5\setminus e]=6.
$$
Hence, $G$ is also $K_5\setminus e$-minor-free and
Proposition \ref{known-cases-conjecture} (i)
implies that $\cut^{\Box}(G)$  is normal.
\end{proof}

We recall the following known results from
\cite[Section 6]{Romer-Madani}
that classify connected graphs with very small regularities.

\begin{proposition}
\label{prop-reg-known-cases}
 Let $G=(V,E)$ be a connected graph with $|E|\geq 1$ and $r=\reg\Kb[G]$. Then:
\begin{enumerate}
\item
$r=0$ if and only if $G=C_3=K_3$ or $G=P_1=K_2$.
\item
$r=1$ if and only if $G=C_3\#_0P_1=K_3\#_0K_2$ or $G=P_2$.
\end{enumerate}
\end{proposition}
(One should note that the notation for $P_n$ differs by one in the index compared to \cite{Romer-Madani}).
\begin{proof}
 The case $r=0$ follows  from \cite[Proposition 3.1, Proposition 3.2]{Romer-Madani}
and the case $r=1$ is obtained from
\cite[Proposition 3.2, Corollary 6.11]{Romer-Madani}. Alternatively, the proof follows also from Lemma \ref{regularity-bipartite} and Theorem \ref{regularity-main}.
\end{proof}

We are ready to state our second main result of this section.
Extending the results of Proposition \ref{prop-reg-known-cases} we describe all connected graphs whose cut algebras have regularities less than or equal to $4$. Recall that in this case $(I_G)_1=0$. Thus, either $I_G=0$ or the ideal is generated in degrees $\geq 2$.
\begin{theorem}
\label{Theorem-classify-reg-4}
Let $G=(V,E)$ be a connected graph with $|E|\geq 1$ and $r=\reg\Kb[G]$. Then:
\begin{enumerate}
\item
$r=2$ if and only if
$G$ is a tree with $|E|=3$
or
$G$ contains a $C_3$  with $|E|=5$.
\item
$r=3$ if and only if $G$ is a tree with $|E|=4$ or $G=C_4$ or $G=C_5$ or  $G$  contains a $C_3$ with $|E|=6$.
\item
$r=4$ if and only if $G$ is bipartite with $|E|=5$
or $G=C_5\#_0 P_1$ or
$G$ contains a $C_3$ with $|E|=7$.
\end{enumerate}
\end{theorem}

\begin{proof}
First note that by Lemma \ref{small-regularity-normal} any cut algebra with regularity $\leq 4$ is normal.
It follows from Lemma \ref{lower-bound} that
$
|E|\leq r+3\leq 7.
$
Conversely, any connected graph with $|E|\leq 7$ is $K_5\setminus e$-minor free
and thus by
Proposition \ref{known-cases-conjecture} (i)
has a normal cut polytope.

In conclusion, we have to consider all possible graphs with $|E|\leq 7$
and have to determine their regularity, where we can use the fact that the cut polytopes are normal. We proceed on a case by case basis with respect to
$|E|\geq 1$.

\emph{Case $|E|=1$}:
Here $G$ has to equal $P_1$ with $r=0$ by Lemma \ref{regularity-bipartite}.
This case is included in Proposition \ref{prop-reg-known-cases}(i)
and not relevant here.

\emph{Case $|E|=2$}:
There exists only the possibility $G=P_2$ with $r=1$ by Lemma \ref{regularity-bipartite}.
This case is included in Proposition \ref{prop-reg-known-cases}(ii)
and not relevant here.

\emph{Case $|E|=3$}: If $G$ is a tree, then $r=2$ by Lemma \ref{regularity-bipartite}. The only other possibility is $G=C_3=K_3$
where $r=0$ by Theorem \ref{regularity-main}, which is again
included in Proposition \ref{prop-reg-known-cases}(i)
and not relevant here.

\emph{Case $|E|=4$}: If $G$ contains a triangle $C_3$, then
$G=C_3\#_0 P_1$ and $r=1$ by Theorem \ref{regularity-main}, which
is included in Proposition \ref{prop-reg-known-cases}(ii)
and not relevant here.
Otherwise $G=C_4$ or $G$ is a tree and in both cases $r=3$ by
Lemma \ref{regularity-bipartite}.

\emph{Case $|E|=5$}:
If $G$ is a tree, then $r=4$ by Lemma \ref{regularity-bipartite}. If $G$ contains a triangle $C_3$, then $r=2$ by Theorem \ref{regularity-main}.
Next assume that $G$ has cycles, but does not contain a $C_3$.
Then either $G=C_5$ with $r=3$ by Lemma \ref{regularity-odd},
or the bipartite graph $G=C_4\#_0P_1$ with $r=4$ by Lemma \ref{regularity-bipartite}.

\emph{Case $|E|=6$}:
If $G$ is a tree, then  $r=5$ by Lemma \ref{regularity-bipartite} and this case is not listed in the theorem. If $G$ contains a triangle $C_3$, then  $r=3$ by Theorem \ref{regularity-main}.

Next assume that $G$ has  cycles, but does not contain a $C_3$. The first case is that $G=C_6$
with $r=5$ by Lemma \ref{regularity-bipartite}, which has to be excluded.
If $G\neq C_6$ contains a $C_5$, then this cycle cannot have a chord since
otherwise there would exists also a triangle. Hence, in this case
$G=C_5\#_0P_1$ with $r=4$ by Theorem \ref{regularity-ring-graph}.

The remaining case is that $G$ contains no  $C_3, C_5$ and $C_6$,
but a $C_4$. Then $G$ is bipartite, because it has no cycles of odd length and thus $r=5$ by Lemma \ref{regularity-bipartite}; a case which has again to be excluded from our desired list.

\emph{Case $|E|=7$}:
If $G$ is a tree, then  $r=6$ by Lemma \ref{regularity-bipartite}. Again a subcase we have to exclude from our list. If $G$ contains a triangle then $r=4$ by Theorem \ref{regularity-main}.
Next assume that $G$ has cycles, but does not contain a $C_3$.
The first case is that $G=C_7$
with $r=5$ by Theorem \ref{regularity-ring-graph}, which has again to be excluded.

If $G\neq C_7$ contains a $C_6$, there exist two subcases.
The graph $G$ could contain a chord of $C_6$
and then $G=C_4\#_1C_4$ with $r=6$ by Lemma \ref{regularity-bipartite}.
Or $C_6$ is an induced cycle of $G$. Then
$G=C_6\#_0 P_1$ with $r=6$ by Lemma \ref{regularity-bipartite}.
Both subcases are not relevant for us since $r>4$.

The next subcase is that $G$ contains no $C_3, C_6$ and $C_7$,
but a $C_5$. There cannot exist chords, since otherwise we would obtain
triangles inside $G$. Then the choices for $G$
are
$C_5\#_{P_2}C_4$
or
$C_5\#_0 P_1\#_0P_1$
or
$C_5\#_0P_2$.
In all these cases $r=5$ by Lemma \ref{regularity-example} and
Theorem \ref{regularity-ring-graph}, which all have to be excluded.

The remaining case is that $G$ contains no $C_3$, $C_5$, $C_6$ and $C_7$,
but a $C_4$. Here $G$ is bipartite with  $r=6$ by Lemma \ref{regularity-bipartite}, which is again too large to be included in the desired list.
\end{proof}

\section{Problems}
\label{Section-7}
In this short section we discuss some research questions, which are of interest for future activities. Related to Section \ref{Section-3} and Section \ref{Section-4} the main challenge is of course Conjecture \ref{Conj-StSu}.

In Section \ref{Section-5} we determined in many cases the canonical module
of cut polytopes like in Theorem \ref{Theorem-canonical-module} in the normal Gorenstein case. Since the case of trees is here included in Case (i) of that theorem, it is a natural question to consider the case of cycles $C_n$ for $n\geq 5$. Cut polytopes of cycles are always normal by Corollary \ref{ring-graph}, but almost never Gorenstein (see Theorem \ref{Ohsugi-Gorenstein}), since only $C_3$ and $C_4$ induce this property for their cut polytopes.
\begin{problem}
Determine the canonical module of $\cut^{\Box}(C_n)$ for any $n\geq 5$.
\end{problem}
Computational evidences using Normaliz \cite{Bruns-Normaliz} suggest
that in the case of cycles $C_n$, $n\geq 5$ the canonical module
is not generated in one degree. However, note that by Theorem \ref{regularity-ring-graph} and using Equation (\ref{Eq-reg})
we know the smallest degree of generators for any $C_n$.

In Section \ref{Section-6} many cases are studied where
one knows the Castelnuovo--Mumford regularity. In particular,
having Theorem \ref{Theorem-classify-reg-4} in mind,
the following would be interesting:
\begin{problem}
Continue the classification of cut polytopes with small regularities.
\end{problem}


\begin{thebibliography}{99}

\bibitem{Barahona-83}
F. Barahona,
\emph{The max-cut problem on graphs not contractible to $K_5$}.
Oper. Res. Lett. {\bf 2} (1983), no. 3, 107--111.
%
\bibitem{Barahona-86}
F. Barahona and A.R. Mahjoub,
\emph{On the cut polytope}.
Math. Programming {\bf 36} (1986), no. 2, 157--173.
%
\bibitem{BoEiNi}
J. B\"ohm, D. Eisenbud, and M.J. Nitsche,
\emph{Decomposition of semigroup algebras}.
Exp. Math. {\bf 21} (2012), no. 4, 385--394.
%
\bibitem{BoJuReRi}
T. Bonato, M. J\"unger, G. Reinelt, and G. Rinaldi, \emph{Lifting and separation procedures for the cut polytope}.
Math. Program. {\bf 146} (2014), no. 1--2, Ser. A, 351--378.
%
\bibitem{Bruns-Normaliz}
W. Bruns and the Normaliz Team,
\emph{Normaliz. Algorithms for rational cones and affine monoids}. Available at
\texttt{https://www.normaliz.uni-osnabrueck.de.} %
%
\bibitem{Bruns-Gubeladze}
W. Bruns and J. Gubeladze,
\emph{Polytopes, rings, and $K$-theory}.
Springer Monographs in Mathematics, Springer, 2009.
%
\bibitem{Bruns-Herzog}
W. Bruns and J. Herzog,
\emph{Cohen-Macaulay rings}.
Cambridge Studies in Advanced Mathematics {\bf 39},
Cambridge University Press 1993.
%
\bibitem{BrPiRo}
W. Bruns, P. Li, and T. R\"omer,
\emph{On seminormal monoid rings}.
J. Algebra {\bf 302} (2006), no. 1, 361--386.
%
\bibitem{ChJKNoRo}
M. Chimani, M. Juhnke-Kubitzke, A. Nover, and T. R\"omer,
\emph{Cut Polytopes of Minor-free Graphs}.
Preprint 2019, \texttt{arXiv:1903.01817}.
%
\bibitem{ChRe}
T. Christof and G. Reinelt,
\emph{Decomposition and parallelization techniques for enumerating the facets of combinatorial polytopes}. Internat. J. Comput. Geom. Appl. {\bf 11} (2001), no. 4, 423--437.
%
\bibitem{Deza-Dutour-16}
M. Deza and M. Dutour Sikiri\'{c},
\emph{Enumeration of the facets of cut polytopes over some highly symmetric graphs}.
Int. Trans. Oper. Res. {\bf 23} (2016), no. 5, 853--860.
%
\bibitem{Deza-Dutour-20}
M. Deza and M. Dutour Sikiri\'{c},
\emph{Generalized cut and metric polytopes of graphs and simplicial complexes}.
Optim. Lett. {\bf 14} (2020), no. 2, 273--289.
%
\bibitem{Deza-Laurent-I}
M. Deza and M. Laurent,
\emph{Applications of cut polyhedra I}.
J. Comput. Appl. Math. {\bf 55} (1994), no. 2, 191--216.
%
\bibitem{Deza-Laurent-II}
M. Deza and M. Laurent,
\emph{Applications of cut polyhedra II}.
J. Comput. Appl. Math. {\bf 55} (1994), no. 2, 217--247.
%
\bibitem{Deza-Laurent-10}
M. Deza and M. Laurent,
\emph{Geometry of cuts and metrics}.
Algorithms and Combinatorics {\bf 15}, softcover printing, Springer 2010.
%
\bibitem{Diestel}
R. Diestel,
\emph{Graph theory}.
Graduate Texts in Mathematics {\bf 173}, fifth edition, Springer, 2017.
%
\bibitem{E}
A. Engstr\"om,
\emph{Cut ideals of $K_4$-minor-free graphs are generated by quadrics}.
Michigan Math. J. {\bf 60} (2011), no. 3, 705--714.
%
\bibitem{Epstein-Nguyen}
N. Epstein and H.D. Nguyen,
\emph{Algebra retracts and Stanley--Reisner rings}.
J. Pure Appl. Algebra {\bf 218} (2014), no. 9, 1665--1682.
%
\bibitem{GeKaRe}
A. Geroldinger, F. Kainrath, and A. Reinhart,
\emph{Arithmetic of seminormal weakly Krull monoids and domains}.
J. Algebra {\bf 444} (2015), 201--245.
%
\bibitem{Hochster}
M. Hochster,
\emph{Rings of invariants of tori, Cohen--Macaulay rings generated by monomials, and polytopes}.
Ann. of Math. (2) {\bf 96} (1972) 318--337.
%
\bibitem{Hochster-Roberts}
M. Hochster and J.L. Roberts,
\emph{The purity of the Frobenius and local cohomology}.
Advances in Math. {\bf 21} (1976), no. 2, 117--172.
%
\bibitem{KW}
V. Kaibel and M. Wolff,
\emph{Simple 0/1-Polytopes}.
European J. Combin. {\bf21} (2000), 139--144.
%
\bibitem{Kaminski}
M. Kami\'{n}ski,
\emph{MAX-CUT and containment relations in graphs}.
Theoret. Comput. Sci. {\bf 438} (2012), 89--95.
%
\bibitem{Karp}
R.M. Karp, \emph{Reducibility among combinatorial problems}. In Proceedings of a symposium on the \emph{Complexity of computer computations}, 85--103, 1972.
%
\bibitem{KNP}
D. Kr\'{a}l and S. Norine, and O. Pangr\'{a}c,
\emph{Markov bases of binary graph models of $K_4$-minor free graphs}.
J. Combinatorial Theory {\bf 117} (2010), 759--765.
%
\bibitem{Laurent}
M. Laurent,
\emph{Hilbert bases of cuts}.
Discrete Math. {\bf 150} (1996), no. 1--3, 257--279.
%
\bibitem{Lason-Michalek}
M. Laso\'{n} and M. Micha{\l}ek,
\emph{A note on seminormality of cut polytopes}.
Preprint 2020.
%
\bibitem{NP}
U. Nagel and S. Petrovi\'{c},
\emph{Properties of cut ideals associated to ring graphs}.
J. Commut. Algebra {\bf 1} (2009), no. 3, 547--565.
%
\bibitem{Neta}
J. Neto,
\emph{On the polyhedral structure of uniform cut polytopes}.
Discrete Appl. Math. {\bf 175} (2014), 62--70.
%
\bibitem{Nguyen}
D.H. Nguyen,
\emph{Seminormality and local cohomology of toric face rings}. J. Algebra {\bf 371} (2012), 536--553.
%
\bibitem{Nitsche}
M.J. Nitsche,
\emph{Castelnuovo--Mumford regularity of seminormal simplicial affine semigroup rings}.
J. Algebra {\bf 368} (2012), 345--357.
%
\bibitem{Ohsugi}
H. Ohsugi,
\emph{Normality of cut polytopes of graphs is a minor closed property}. Discrete Math. {\bf 310} (2010), no. 6--7, 1160--1166.
%
\bibitem{Ohsugi-Gorenstein}
H. Ohsugi,
\emph{Gorenstein cut polytopes}.
European J. Combin. {\bf 38} (2014), 122--129.
%
\bibitem{Potka-Sarmiento}
S. Potka and C. Sarmiento,
\emph{Betti numbers of cut ideals of trees}.
J. Algebr. Stat. {\bf 4} (2013), no. 1, 108--117.
%
\bibitem{Romer-Madani}
T. R\"omer and S. Saeedi Madani,
\emph{Retracts and algebraic properties of cut algebras}.
European J. Combin. {\bf 69} (2018), 214--236.
%
\bibitem{Sakamoto}
R. Sakamoto,
\emph{Lexicographic and reverse lexicographic quadratic Gr\"obner bases of cut ideals}.
To appear in  J. Symbolic Comput.,  \texttt{arXiv:1802.08796}.
%
\bibitem{Shibata}
K. Shibata,
\emph{Strong Koszulness of the toric ring associated to a cut ideal}. Comment. Math. Univ. St. Pauli {\bf 64} (2015), no. 1, 71--80.
%
\bibitem{Stanley}
R.P. Stanley,
\emph{Hilbert functions of graded algebras}.
Advances in Math. {\bf 28} (1978), no. 1, 57--83.
%
\bibitem{Sturmfels-Sullivant}
B. Sturmfels and S. Sullivant,
\emph{Toric geometry of cuts and splits}.
Michigan Math. J. {\bf 57} (2008), 689--709.
%
\bibitem{sullivant}
S. Sullivant,
\emph{Compressed polytopes and statistical disclosure limitation}.
Tohoku Math. J. (2) {\bf 58} (2006), no. 3, 433--445.
%
\bibitem{Swan}
R. G. Swan,
\emph{On seminormality}.
J. Algebra {\bf 67} (1980), no. 1, 210--229.
%
\bibitem{Traverso}
C. Traverso,
\emph{Seminormality and Picard group}.
Ann. Scuola Norm. Sup. Pisa Cl. Sci. (3) {\bf 24} (1970), 585--595.
%
\bibitem{Yanagawa}
K. Yanagawa,
\emph{Dualizing complexes of seminormal affine semigroup rings and toric face rings}.
J. Algebra {\bf 425} (2015), 367--391.
\end{thebibliography}
\end{document}